\documentclass[11pt]{article}
\usepackage{amsmath,amsthm,amsfonts,amssymb,amscd, amsxtra, mathrsfs,enumitem, mathtools}
\usepackage{url}

\usepackage[margin=2.3 cm,nohead]{geometry}
\usepackage{color}
\usepackage{pdfsync}
\usepackage{hyperref}
\usepackage{graphicx}
\usepackage{grffile}
\graphicspath{{/}}
\usepackage{subcaption}
\usepackage{caption}
\usepackage{booktabs}    
\usepackage{float}  
\usepackage[
  backend=biber,
  style=ieee,
  url=false,
  isbn=false,
  sortcites=true
]{biblatex}

\usepackage{multicol}
\usepackage{multirow}
\usepackage{siunitx}
\setlist[enumerate]{label=(\roman*), align=left}
\newtheorem{theorem}{Theorem}
\newtheorem{lemma}[theorem]{Lemma}
\newtheorem{definition}{Definition}

\newtheorem{remark}{Remark}

\newtheorem{assumption}{Assumption}

\usepackage{graphicx,float}

\usepackage{algorithm}
\usepackage{algpseudocode}

\newcommand{\R}{\mathbb{R}}

\usepackage{amssymb}
\usepackage{relsize}

\algnewcommand{\Input}[1]{%
  \State \textbf{Input:} {\raggedright #1}
  
}

\algnewcommand{\Initialize}[1]{%
  \State \textbf{Initialize:}
  \Statex \hspace*{\algorithmicindent}\parbox[t]{.8\linewidth}{\raggedright #1}
}

\algnewcommand{\Output}[1]{%
  \State \textbf{Output:} {\raggedright #1}
}

\newcommand{\algoO}{\textbf{\smaller IPCMAS1}~}
\newcommand{\algoT}{\textbf{\smaller IPCMAS2}~}
\newcommand{\algoDIM}{\textbf{\smaller DIPCM}~}
\newcommand{\algoD}{\textbf{\smaller DeyHICPP}~}
\newcommand{\algoS}{{\textbf{\smaller SICIP}~}}

\begin{document}
\title{Projection and contraction methods with double inertial steps  for
 variational inclusion problems on Hilbert spaces}

\author{
Moin  Uddin\thanks{Interdisciplinary Research Center for Smart Mobility and Logistics, King Fahd University of Petroleum \& Minerals, Dhahran, 31261, Saudi Arabia, (e-mail:{\tt umoin2723@gmail.com}).}
\and 
Mohammed Alshahrani \thanks{Corresponding author. Department of Mathematics, King Fahd University of Petroleum \& Minerals, Dhahran, 31261, Saudi Arabia\\
Interdisciplinary Research Center for Smart Mobility and Logistics, King Fahd University of Petroleum \& Minerals, Dhahran, 31261, Saudi Arabia, (e-mail:{\tt mshahrani@kfupm.edu.sa}).}
\and
Qamrul Hasan Ansari \thanks{Department of Mathematics, King Fahd University of Petroleum \& Minerals, Dhahran, 31261, Saudi Arabia\\
Interdisciplinary Research Center for Smart Mobility and Logistics, King Fahd University of Petroleum \& Minerals, Dhahran, 31261, Saudi Arabia, (e-mail:{\tt qhansari@gmail.com}).}
}

\maketitle

\begin{abstract}
{
In this paper, we propose three projection--contraction algorithms for solving variational inclusion problems in the setting of Hilbert spaces, each incorporating a double inertial technique into the projection--contraction framework. The first algorithm \algoO achieves weak convergence under monotonicity and Lipschitz continuity of the single-valued operator, with an adaptive stepsize rule that does not require prior knowledge of the Lipschitz constant. A modified variant \algoT of \algoO attains $R$-linear convergence under the strong monotonicity assumption. The third algorithm \algoDIM obtains strong convergence to the minimum-norm solution without requiring strong monotonicity. We illustrate the proposed methods on an abstract variational inclusion problem and apply them to the split feasibility problem and the elastic net regularization problem, comparing them with existing methods in the literature. The $R$-linear convergence of \algoT is also verified through numerical simulations.
}
\end{abstract}


\medskip



\section{Introduction}


The variational inclusion problem unifies a range of mathematical models,
including variational inequalities, convex minimization, split feasibility,
and equilibrium problems. It is widely used in machine learning, signal and
image processing, and inverse problems;
See, for example, \parencite{combettes_signal_2005, duchi_efficient_2009, raguet_generalized_2013}, and the references therein.

Let \(\mathscr{H}\) be a Hilbert space  with the inner product 
\(\langle \cdot, \cdot \rangle\) and the induced norm \(\|\cdot\|\). 
Consider a set-valued operator \(A \colon \mathscr{H} \to 2^{\mathscr{H}}\) 
and a single-valued operator \(B \colon \mathscr{H} \to \mathscr{H}\). 
The \emph{variational inclusion problem} (in short, VIP) consists in 
finding a point \(x \in \mathscr{H}\) such that
\begin{equation}\label{mainproblem}
   {\mathbf{0}} \in A(x) + B(x). 
\end{equation}
We denote by \(\Omega := (A + B)^{-1}({\mathbf{0}})\) the solution set of this problem. 
When $B$ is a zero operator, the problem \eqref{mainproblem} reduces to the classical
monotone inclusion studied in \cite{rockafellar_monotone_1976}, which is central to
convex optimization and variational analysis.

One of the most fundamental {algorithms} for solving the variational inclusion problem (VIP) 
is the forward--backward splitting method (FBSM), originally proposed by Lions 
\parencite{lions_splitting_1979} that generates iterates according to
\[
x_{n+1} = J^{A}_{\lambda}\big(x_n - \lambda B(x_n)\big),
\]
where $\lambda > 0$ is an appropriate stepsize and $J^{A}_{\lambda} = (I + \lambda A)^{-1}$
denotes the resolvent of the operator $A$. The convergence of FBSM is guaranteed under the maximal monotonicity of $A$ and the restrictive assumption that the operator $B$ is $\vartheta$-inverse strongly monotone (or $\vartheta$-cocoercive); that is, there exists $\vartheta > 0$ such that
\[
\langle B(x) - B(y),\, x - y \rangle \geq \vartheta \|B(x) - B(y)\|^2, 
\quad \forall x,y \in \mathscr{H}.
\]
This cocoercivity condition is restrictive in many applications,
limiting the applicability of the classical forward--backward approach.

To weaken this requirement, Tseng \cite{tseng_modified_2000} proposed a refined variant 
known as the forward--backward--forward (FBF) algorithm, which eliminates the need of 
cocoercivity and requires only the Lipschitz continuity of $B$. 
This comes at the cost of an additional forward evaluation per iteration. 
Tseng's iteration is given by
\begin{equation}\label{FBFMethod}
\begin{cases}
y_n = J^{A}_{\lambda_n}(x_n - \lambda_n B(x_n)), \\
x_{n+1} = y_n - \lambda_n \big(B(y_n) - B(x_n)\big),
\end{cases}
\quad n \in \mathbb{N},
\end{equation}
which converges weakly to a solution of the VIP provided that $B$ is $L$-Lipschitz continuous 
and $\lambda_n \in (0, 1/L)$. Numerous extensions of the FBF method have been studied
{in the literature}; 
See, for example, \parencite{bot_inertial_2016, bot_relaxed_2023, cholamjiak_relaxed_2021, ofem_relaxed_2024},
and the references therein.

Another important method that avoids the cocoercivity assumption is 
the following projection--contraction method (PCM) proposed by 
Zhang and Wang \cite{zhang_proximal_2018}:
\begin{align}\label{PCM}
y_n &= J^{A}_{\lambda_n}(x_n - \lambda_n B(x_n)), \notag\\
x_{n+1} &= x_n - \gamma\, \eta_n\, d_n,
\end{align}
where
\[
d_n = (x_n - y_n) - \lambda_n\big(B(x_n) - B(y_n)\big),
\]
and
\[
\eta_n =
\begin{cases}
\dfrac{\langle x_n - y_n,\, d_n \rangle}{\|d_n\|^2}, & \text{if } d_n \neq 0, \\
0, & \text{otherwise}.
\end{cases}
\]
where $\gamma \in (0, 2)$ is a relaxation parameter, and the stepsizes $\lambda_n$ 
are chosen according to the rules in \cite{zhang_proximal_2018}. 
They established weak convergence of the sequence $\{x_n\}$ to a solution of the VIP 
under suitable assumptions.  

Accelerating convergence is important for computing accurate solutions to practical problems.
The inertial technique, which comes from the discrete approximation of second-order dissipative
dynamical systems~\parencite{attouch_heavy_2000, attouch_asymptotic_2002}, is now a common tool
for speeding up the convergence of iterative algorithms. By adding momentum-type terms based on
previous iterations, inertial methods often give faster practical performance. They have been
widely applied in convex and nonconvex optimization, monotone inclusion problems, and variational
inequalities;
see, for example,~\parencite{wang_inertial_2022, dong_inertial_2018, bot_relaxed_2023, bot_inertial_2016}.

{ In particular, several inertial splitting methods have been proposed for variational inclusions and related split-type problems; for instance, Peeyada et al.~\cite{peeyadaInertialMannForwardbackward2022} studied an inertial Mann forward--backward splitting algorithm; Van Hieu et al.~\cite{vanhieuIterativeRegularizationMethods2022} proposed iterative regularization methods with new stepsize rules; Nguyen and Kim Hoa~\cite{nguyenInertialReflectedforwardbackwardSplitting2026} introduced an inertial reflected-forward-backward splitting method with improved step sizes; Reich et al.~\cite{reichVariationalInequalitiesSolution2023} studied variational inequalities over the solution sets of split variational inclusion problems; Wang~\cite{Wang2022a} investigated the split feasibility problem with multiple output sets for demicontractive mappings; and Ezeafulukwe et al.~\cite{ezeafulukweMonotoneVariationalInclusion2025} proposed a modified splitting approach for monotone variational inclusion problems.}

{
Recently, Dey~\cite{dey_hybrid_2023} introduced a hybrid inertial and contraction proximal point algorithm for solving~\eqref{mainproblem}. Specifically, Dey's algorithm generates iterates according to
\begin{equation}\label{eq:dey_algorithm}
 \begin{aligned}
 \zeta_n &=
\begin{cases}
\min\left\{\zeta, \dfrac{\tau_n}{\|x_n - x_{n-1}\|}\right\}, & \text{if } x_n \neq x_{n-1}, \\
\zeta, & \text{otherwise},
\end{cases} \\[2mm]
w_n &= x_n + \zeta_n (x_n - x_{n-1}), \\[1mm]
y_n &= J^{A}_{\lambda_n}(w_n - \lambda_n B(w_n)), \\[1mm]
z_n &= w_n - \gamma \eta_n d_{n}, \\[1mm]
x_{n+1} &= (1 - \theta_n - \delta_n)x_n + \delta_n z_n.
\end{aligned}
\end{equation}
where $\zeta>0, \gamma\in(0,2)$, the stepsize $\lambda_{n}\subset[a,b] \subset (0, 1/L)$, $d_n$ and $\eta_n$ are defined as in the PCM framework, and $z_n$ is an auxiliary contraction step. Dey established strong convergence of this scheme under the maximal monotonicity of $A$ and the $L$-Lipschitz continuity and monotonicity of~$B$.
Dey's method has two limitations: (i)~the stepsize $\lambda_{n}$ depends explicitly on the Lipschitz constant $L$ of $B$, which is often unknown or difficult to compute in practice; and (ii)~only a single inertial extrapolation step is employed.

Several works have explored the use of double inertial extrapolation terms to improve convergence behavior. For example, Suantai et al.~\cite{suantaiForwardBackwardForward2023} proposed a double inertial forward--backward--forward algorithm for monotone inclusions, in which two inertial terms are incorporated into the Tseng method \cite{tseng_modified_2000} together with an adaptive stepsize rule.

Moreover, Suantai et al.~\cite{Suantai2024} introduced a projection--contraction algorithm \algoS that incorporates two inertial extrapolation terms. Specifically, the \algoS\ algorithm generates iterates according to
\begin{equation}\label{eq:suantai2024_algorithm}
\begin{cases}
w_n = x_n + \vartheta_{n}(x_n - x_{n-1}) + \varsigma_n(x_{n-1} - x_{n-2}), \\
y_n = J^{A}_{\lambda_n}(w_n - \lambda_n B(w_n)), \\
d_{n} = (w_n - y_n) - \lambda_n(B(w_n) - B(y_n)), \\
x_{n+1} = w_n - \gamma \eta_n\, d_{n},
\end{cases}
\end{equation}
where $\eta_n = \langle w_n - y_n, d_{n}\rangle / \|d_{n}\|^2$, and $\{\vartheta_n\}$, $\{\varsigma_n\}$ are nonnegative sequences satisfying $\sum\limits_{n=1}^{\infty} \vartheta_n < +\infty$ and $\sum\limits_{n=1}^{\infty} \varsigma_n < +\infty$. The stepsize is updated adaptively via $\lambda_{n+1} = \min\{\mu\|w_n - y_n\|/\|B(w_n) - B(y_n)\|,\, \lambda_n\}$, avoiding dependence on the Lipschitz constant. They established weak convergence of the algorithm under monotonicity assumptions, with applications to data classification.
The \algoS\ algorithm combines its two inertial terms in a single extrapolation step using three consecutive iterates ($x_n, x_{n-1}, x_{n-2}$), and the summability conditions on $\{\vartheta_{n}\}$ and $\{\varsigma_{n}\}$ force the inertial effect to diminish over iterations. In contrast, the method proposed in this paper employs two \emph{independent} inertial steps with separate parameters that are not required to decay and provides stronger convergence guarantees, including strong convergence and
R-linear convergence.
}


Motivated by recent inertial-acceleration techniques, this paper introduces a double-inertial projection--contraction method for solving variational inclusion problems. Our primary contributions are:
\begin{enumerate}
    \item {We propose a double-inertial projection--contraction algorithm \algoO that incorporates two separate inertial extrapolation steps into the projection--contraction framework. Unlike the \algoS\ approach~\cite{Suantai2024}, where two inertial terms are combined in a single extrapolation step using three consecutive iterates ($x_k, x_{k-1}, x_{k-2}$), our method applies two independent inertial steps to different components of the iterate update:
    \[
    z_n = x_n + \beta(x_n - x_{n-1}), \qquad w_n = x_n + \theta_n(x_n - x_{n-1}),
    \]
    followed by $x_{n+1} = (1-\alpha)z_n + \alpha u_n$, where $u_n$ is the projection--contraction correction based on $w_n$. This decoupling allows the acceleration mechanism (via $z_n$) and the projection--contraction correction (via $w_n$) to be controlled independently through separate parameters $\beta$ and $\theta_n$. We also develop an adaptive stepsize strategy that removes any dependence on the Lipschitz constant $L$ of $B$, improving on Dey's algorithm~\cite{dey_hybrid_2023}, which requires $\lambda \in (0, 1/L)$.}

    \item We establish weak convergence of the proposed algorithm without requiring prior knowledge of the Lipschitz constant. Under the strong monotonicity assumption on operator $A$, we prove strong convergence to the unique solution of the variational inclusion problem. A modified variant \algoT is shown to enjoy $R$-linear convergence under the same hypothesis.

    \item {We further introduce a third algorithm \algoDIM that couples the projection--contraction step with a three-term Halpern--Mann iteration. Under monotonicity and Lipschitz continuity of $B$ alone---without the strong monotonicity assumption---we prove that the iterates of \algoDIM converge strongly to $P_\Omega(\mathbf{0})$, the minimum-norm element of the solution set.}

    \item {Numerical experiments on abstract variational inclusion problems compare our double-inertial scheme with the single-inertial \algoD\ algorithm~\cite{dey_hybrid_2023} and the double-inertial \algoS\ algorithm~\cite{Suantai2024}. Controlled comparisons with synchronized inertial parameters isolate the structural benefits of our approach, and a sensitivity analysis of the key algorithmic parameters is provided. We also apply the method to the split feasibility problem and elastic net regularization. Numerical simulations confirm the $R$-linear convergence of \algoT\ in practice.}
\end{enumerate}

{This paper is organized as follows. Section~\ref{S:2} collects basic definitions and lemmas used in the convergence analysis. Section~\ref{S:3} presents the three proposed algorithms and their convergence theory: \algoO\ with weak convergence under monotonicity and Lipschitz continuity of $B$, without requiring the Lipschitz constant; strong convergence of \algoO\ under the strong monotonicity assumption; the modified variant \algoT\ with $R$-linear convergence under the same assumption; and \algoDIM, which couples the projection--contraction step with a three-term Halpern--Mann iteration to obtain strong convergence to $P_\Omega(\mathbf{0})$ without strong monotonicity. Section~\ref{S:5} reports numerical experiments comparing the three proposed algorithms with the methods of Dey~\cite{dey_hybrid_2023} and Suantai et al.~\cite{Suantai2024} on a classical variational inclusion problem, a split feasibility problem in $L^2[0,1]$, and an elastic net regularization problem; the $R$-linear convergence of \algoT\ is also verified empirically. The paper concludes in Section~\ref{S:6}.}


\section{Preliminaries}\label{S:2}


In this section, we recall some basic definitions and results that will be used 
in the convergence analysis of our method. 
Let $\{x_n\}$ be a sequence in $\mathscr{H}$. 
We write $x_n \rightharpoonup x$ (respectively, $x_n \to x$) to denote the weak 
(respectively, strong) convergence of $\{x_n\}$ to $x \in \mathscr{H}$ as $n \to \infty$. 

Let us begin with some concepts related to the monotonicity of an operator.

\begin{definition}\label{def:lipschitz}
A single-valued operator $B\colon \mathscr{H} \to \mathscr{H}$ is called
\begin{itemize}
    \item[{\rm(i)}] \textit{monotone} if
    $$\langle B(x) -B(y),\ x - y \rangle \geq 0, \quad\forall x, y \in \mathscr{H};$$
    
    \item[{\rm(ii)}] \textit{$\nu$-strongly monotone} if there exists  $\nu > 0$ such that
    $$\langle B(x) - B(y),\ x - y \rangle \geq \nu \|x - y\|^2,
    \quad \forall x, y \in \mathscr{H};$$
    
    \item[{\rm(iii)}] \textit{$L$-Lipschitz continuous} if there exists  $L > 0$ such that
    $$\|B(x) - B(y)\| \leq L\,\|x - y\|, \quad\forall x, y \in \mathscr{H}.$$
\end{itemize}
\end{definition}

Let $A : \mathscr{H} \to 2^{\mathscr{H}}$ be a set-valued operator. 
The \textit{graph} of $A$ is the subset of $\mathscr{H} \times \mathscr{H}$ defined by
\[
\operatorname{Graph}(A) := \{(x, u) \in \mathscr{H} \times \mathscr{H} : u \in A(x)\}.
\]

\begin{definition}\label{def:maximal_monotone}{\rm\cite{bauschke_convex_2017}}
A set-valued operator $A \colon \mathscr{H} \to 2^\mathscr{H}$ is said to be 
\begin{itemize}
    \item[{\rm(i)}] \textit{monotone} if 
    $$  \langle u - v, x - y \rangle \geq 0, \quad \forall x, y \in \mathscr{H}
    \mbox{ and } \forall u \in A(x), \, v \in A(y);$$

    \item[{\rm(ii)}] \textit{maximally monotone} if it is monotone and for any given 
    $(x, u) \in \mathscr{H} \times \mathscr{H}$,
    $$ \langle u - v, x - y \rangle \geq 0, \, \, \forall  (y, v) \in \operatorname{Graph}(A)
    \mbox{ implies } u \in A(x);$$

    \item[{\rm(iii)}]  \textit{$\nu$-strongly monotone} if there exists $\nu > 0$ 
    such that 
    $$ \langle u - v, x - y \rangle \geq \nu \|x - y\|^2, \quad \forall x, y \in \mathscr{H}, 
    \mbox{ and }  u \in A(x), \, v \in A(y).
    $$
\end{itemize}
\end{definition}

Let $A \colon \mathscr{H} \to 2^{\mathscr{H}}$ be a set-valued maximal monotone operator. 
{Recall that, for $\lambda > 0$, the resolvent of $A$ is a single-valued operator
$J^{A}_{\lambda}\colon\mathscr{H} \to \mathscr{H}$ defined by
\[
J^{A}_{\lambda}(x) =(I+\lambda A)^{-1}(x),\quad \forall x\in \mathscr{H},
\]
where $I$ is the identity operator on $\mathscr{H}$.} We have the following lemmas.

\begin{lemma}{\rm(\cite[{Lemma 1}]{gibali_tseng_2018})}\label{fiexdlemma}
Let $A \colon \mathscr{H} \to 2^{\mathscr{H}}$  be a maximal monotone operator 
and $B :\mathscr{H} \to \mathscr{H}$ be a single-valued  operator.  
For each $\lambda > 0$, define
\(
T_\lambda(x) := {J^{A}_{\lambda}}(x - \lambda B(x)) 
\) for all $x\in\mathscr{H}$.
Then,
$x \in (A + B)^{-1}({\mathbf{0}})$  if and only if  $x\in \operatorname{Fix}(T_\lambda),$
where $\operatorname{Fix}(T_\lambda) := \{x \in \mathscr{H} : T_\lambda (x)=x\}.$ 
\end{lemma}

\begin{lemma}\label{lemma:sum_maximal_monotone}{\rm (\cite[{Corollary 20.28}]{bauschke_convex_2017})}
Let $A \colon \mathscr{H} \to 2^{\mathscr{H}}$ be a maximal monotone operator and 
$B \colon \mathscr{H} \to \mathscr{H}$ be a monotone and {$L$-}Lipschitz continuous operator. 
Then the set-valued operator $A + B$ is maximal monotone.
\end{lemma}

\begin{lemma}\label{lemma:inner_product_identity}
For any  $x, y \in \mathscr{H}$ and $\kappa\in\mathbb{R}$, the following identities hold: 
\begin{itemize}
\item[{\rm(i)}] 
    $2\langle x, y \rangle = \|x\|^2 + \|y\|^2 - \|x - y\|^2 = \|x + y\|^2 - \|x\|^2 - \|y\|^2.$
    
    \item[{\rm(ii)}] 
    $\|\kappa x+(1-\kappa)y\|^2=\kappa \|x\|^2+(1-\kappa)\|y\|^2-\kappa(1-\kappa)\|x-y\|^2.$

    \item[{\rm(iii)}] 
    $\|x+y\|^{2}\leq\|x\|^{2}+2\,\left\langle x+y, y\right\rangle.$
   
\end{itemize}
\end{lemma}

\begin{lemma}{\rm(\cite[{Lemma 2.2}]{mainge_convergence_2008})}\label{lemma:auxiliaryconvergence}
Let $\{\varphi_n\}$, $\{\delta_n\}$, and $\{\zeta_n\}$ be sequences in $[0, +\infty)$ 
such that
\begin{equation*}
\varphi_{n+1} \leq \varphi_n + \zeta_n (\varphi_n - \varphi_{n-1}) + \delta_n, \quad \forall n 
\geq 1, \quad \sum_{n=1}^{\infty} \delta_n < +\infty,
\end{equation*}
and suppose that there exists a real number $\zeta$  with $0 \leq \zeta_n \leq \zeta<1$ for all 
$n \in \mathbb{N}$. Then the following conclusions hold:
\begin{enumerate}
\item [{\rm(i)}]$\sum_{n=1}^{\infty} [\varphi_n - \varphi_{n-1}]_{+}<{+}\infty$, 
where $[m]_+ := \max\{m, 0\}$;

\item [{\rm(ii)}] there exists $\varphi^* \in [0, {+}\infty)$ such that 
$\displaystyle \lim\limits_{n \to \infty} \varphi_n = \varphi^*$.
\end{enumerate}
\end{lemma}

{

\begin{lemma}{\rm(\cite[Lemma 2.6]{Saejung2012})}\label{lmm:convergence}
Let $\{a_n\}$ be a nonnegative real sequence, $\{\alpha_n\} \subset (0,1)$ with
$\sum\limits_{n=1}^{\infty} \alpha_n = \infty$, and let $\{b_n\}$ be a real sequence. Assume that
$a_{n+1} \le (1-\alpha_n)a_n + \alpha_n b_n$ for all $n \ge 1$. If
$\limsup\limits_{k\to\infty} b_{n_k} \le 0$ for every subsequence $\{a_{n_k}\}$ satisfying
$\liminf\limits_{k\to\infty}(a_{n_k+1} - a_{n_k}) \ge 0$, then $\lim\limits_{n\to\infty} a_n = 0$.
\end{lemma}
}

\begin{lemma}{\rm(\cite[{ Lemma 2.47}]{bauschke_convex_2017})}\label{opialcondition}
Let $C$ be a nonempty subset of $\mathscr{H}$ and $\{x_n\}$ be a sequence in $\mathscr{H}$ 
such that $\lim\limits_{n \to \infty} \|x_n - x\|$ exists for all $x \in C$.
If a subsequence of $\{x_n\}$ converges weakly to $x \in C$, 
then the sequence $\{x_n\}$ converges weakly to a point in $C$.
\end{lemma}


\section{Algorithms and Their Convergence}\label{S:3}


{In this section we introduce three projection--contraction algorithms with double inertial extrapolation. The basic algorithm \algoO, presented after the standing assumptions below, is studied in the next two subsections (weak convergence, then strong convergence under strong monotonicity). A modified variant \algoT, yielding $R$-linear convergence under strong monotonicity, is introduced in Section~\ref{S:4}. The third algorithm \algoDIM, which couples the projection--contraction step with a three-term Halpern--Mann iteration, is introduced in Section~\ref{S:halpern}.}
Throughout this section, we assume that the following assumptions hold.

\begin{assumption}\label{assumption}
\begin{itemize}
    \item[\textnormal{(A1)}] \label{ass:A1} $A\colon\mathscr{H}\to2^{\mathscr{H}}$ 
    is maximal monotone.
    
    \item[\textnormal{(A2)}] \label{ass:A2} $B\colon\mathscr{H}\to\mathscr{H}$ 
    is $L$-Lipschitz continuous and monotone.
    
    \item[\textnormal{(A3)}] \label{ass:A3} The solution set $\Omega \neq \emptyset$.
   
    \item[\textnormal{(A4)}] \label{ass:A4} $\beta \in \left[0, \frac{1 - 3\alpha}{3(1 - \alpha)}\right)$, and $0 \leq \theta_n \leq \theta_{n+1} \leq 1$.
    
    \item[\textnormal{(A5)}] \label{ass:A5} $0 < \alpha < \frac{1}{3}$.
\end{itemize}
\end{assumption}

\begin{algorithm}[H]\label{A:1}
	\begin{footnotesize}
		\begin{description}
			  \item[Step 0. ]
      Choose 
      $\mu \in (0,1)$,
      $\gamma \in (0,2)$,
      $\lambda_{1} > 0,$ and a nonnegative sequence $\{\xi_{n}\}$ such that $\sum\limits_{n=1}^{\infty} \xi_n<{+}\infty.
      $
      Choose starting points $x_{0}, x_{1} \in \mathscr{H}$ and set $n := 1$.

			\item[Step 1.]  Compute

\begin{equation}\label{defznwnyn}
\begin{cases}
z_n = x_{n}+ \beta(x_n - x_{n-1})  \\
w_{n}=x_{n}+\theta_{n}\left(x_{n}-x_{n-1} \right)\\
y_{n}=J^{A}_{\lambda_{n}}\left(w_n - \lambda_n B(w_n)\right)
\end{cases}
\end {equation}
If $y_{n}=w_{n}$, then stop $y_{n}$ is a solution of problem \eqref{mainproblem}. Otherwise go to \textbf{Step~2}.
\item[Step 2.] Compute $u_{n}=w_{n}-\gamma\eta_{n}d_{n},$ where
\begin{equation}\label{defdn}
 d_{n}:=w_{n}-y_{n}-\lambda_{n}\left(B(w_{n})-B(y_{n})\right),
 \end{equation}
  \textit{and} \begin{equation}\label{defrhon}
            \eta_{n}:=
	\begin{cases}
		\frac{\left\langle w_{n}-y_{n},d_{n}\right\rangle}{\|d_{n}\|^2},& \mbox{if}~ d_{n}\ne0,\\[5pt]
		0 &\mbox{if}~ d_{n}=0.
	\end{cases}
        \end{equation}
                 
                    \item[Step 3.] Compute
                \begin{equation}\label{defxn}
                x_{n+1}=(1-\alpha)z_{n}+\alpha u_{n}.
                \end{equation}

            Update \begin{equation}\label{deflambda}
	\lambda_{n+1}=
	\begin{cases}
		\min\left\{\frac{\mu \|w_{n}-y_{n}\|}{\|B(w_{n})-B(y_{n})\|},
		\lambda_{n}+\xi_{n}\right\} & \mbox{if } B(w_{n})\neq B(y_{n}),\\[5pt]
		\lambda_{n}+\xi_{n} & \mbox{otherwise}.
	\end{cases}
\end{equation}
                Set $n:=n+1$ and go to  \textbf{Step~1}.
		\end{description}
		
		\caption{New Inertial Projection and Contraction Method with Adaptive Step-size (\algoO\!\!)}
       \label{algo:IPCMAS1}
	\end{footnotesize}
\end{algorithm}

\begin{remark}
The proposed algorithm \algoO differs from the algorithm introduced in \eqref{eq:dey_algorithm}
by several essential aspects.

Firstly, algorithm \algoO incorporates an \emph{adaptive stepsize strategy} for the sequence 
$\{\lambda_n\}$, which is updated at each iteration by a simple computation based on the 
previous iterates $w_n$ and $y_n$. 
Unlike the method in \eqref{eq:dey_algorithm}, where the stepsize is required to be chosen 
\emph{a priori} from a fixed interval $(0,1/L)$ depending explicitly on the Lipschitz constant 
$L$, the adaptive rule in \algoO does not require the knowledge of $L$ 
as an input parameter. 
No line search is required.

Secondly, \algoO uses two inertial extrapolation steps through the sequences
$\{z_n\}$ and $\{w_n\}$, so the inertial step and the projection--contraction
correction are decoupled.
This decoupling permits independent control of the acceleration and the stability of the iterates,
in contrast with the single-inertial scheme in \eqref{eq:dey_algorithm}.

Thirdly, the iterative update in \algoO is formulated as a 
\emph{simple convex combination} of the inertial point and the corrected point, 
involving fewer control parameters. 
In contrast, the method in \eqref{eq:dey_algorithm} relies on a more restrictive contraction-type 
update rule with multiple parameters. 
As a result, \algoO has a simpler structure, which simplifies both the analysis
and the implementation.
\end{remark}

{
\begin{remark}[Role of the double inertial structure]\label{remark:double_inertial}
The decoupled double inertial mechanism in Algorithm~\ref{algo:IPCMAS1} is not merely additional momentum: the two extrapolation steps act on different components of the iterate update:
\begin{itemize}
\item The sequence $\{z_n\}$, governed by the parameter $\beta$, provides inertial acceleration in the \emph{averaging step} $x_{n+1} = (1-\alpha)z_n + \alpha u_n$. This term directly influences the final iterate and controls the size of the extrapolation step.
\item The sequence $\{w_n\}$, governed by the parameter $\theta_n$, provides inertial acceleration in the \emph{resolvent evaluation} $y_n = J^A_{\lambda_n}(w_n - \lambda_n B(w_n))$ and the subsequent projection--contraction correction $u_n$.
\end{itemize}
By separating these roles, the parameters $\beta$ and $\theta_n$ can be tuned independently: $\theta_n$ controls the speed of the forward-backward step, while $\beta$ controls the iterate momentum. In the single-inertial framework of~\cite{dey_hybrid_2023}, both roles are handled by a single parameter.
\end{remark}
}

The following lemma shows that $\lambda_n$ in \eqref{deflambda} is well defined.

\begin{lemma}\label{lemma:Lipschitz_lambda_convergence} {\rm\cite{liu_weak_2020}}
The sequence $\{\lambda_{n}\}$ generated by \eqref{deflambda} is well defined and 
$\lim\limits_{n \to \infty} \lambda_n = \lambda$  with  
$\lambda \in \left[ \min\left\{ \lambda_1, \frac{\mu}{L} \right\}, \lambda_1 + \xi \right],$
where $\xi= \sum\limits_{n=1}^{\infty} \xi_n$. Moreover,
\begin{equation}
\|B(w_{n}) - B(y_{n})\| \leq \frac{\mu}{\lambda_{n+1}} \|w_n - y_n\|.
\end{equation}
\end{lemma}

\begin{remark}\label{rm:wnyn}
If $w_{n}=y_{n}$ or $\eta_{n}=0$ in Algorithm \ref{algo:IPCMAS1}, then $y_{n}\in \Omega$.
\end{remark}

It follows from \eqref{defdn} and Lemma \ref{lemma:Lipschitz_lambda_convergence}  that	
\begin{equation}\label{lowerboundofdn}
 		\|d_{n}\|\ge\|w_{n}-y_{n}\|-\lambda_{n}~\|B(w_{n})-B(y_{n})\|
        \ge\left(1-\mu\frac{\lambda_{n}}{\lambda_{n+1}}\right)\|w_{n}-y_{n}\|, 
\end{equation}
and
 
\begin{equation}\label{upperboundofdn}
\|d_{n}\|\le\|w_{n}-y_{n}\|+\lambda_{n}~\|B(w_{n})-B(y_{n})\|
\le\left(1+\mu\frac{\lambda_{n}}{\lambda_{n+1}}\right)\|w_{n}-y_{n}\|.
\end{equation}
From \eqref{lowerboundofdn} and \eqref{upperboundofdn}, we get 
\begin{equation}
 	\left(1-\mu\frac{\lambda_{n}}{\lambda_{n+1}}\right)\|w_{n}-y_{n}\|\le \|d_{n}\|\le\left(1+\mu\frac{\lambda_{n}}{\lambda_{n+1}}\right)\|w_{n}-y_{n}\|.
\end{equation}
Hence, $w_{n}=y_{n}$ if and only if  $d_{n}=0.$ 
If $w_{n}=y_{n}$, then $y_{n}=J^{A}_{\lambda_{n}}\left(w_n - \lambda_n B(w_n)\right)$. 
Thus, by Lemma \ref{fiexdlemma}, we have that $y_{n}\in\Omega.$

\subsection{Weak Convergence}

In this subsection, we prove the weak convergence of the sequence $\{x_{n}\}$ 
generated by Algorithm \ref{algo:IPCMAS1} under Assumption \ref{assumption}.

\begin{lemma}\label{lemma:boundedness}
Let $\{x_{n}\}$ be a sequence generated by Algorithm \ref{algo:IPCMAS1} and $x^\ast\in \Omega$. 
Then the sequence $\{x_{n}\}$ is bounded and $\lim\limits_{n \to \infty} \|x_n - x^\ast\|$ exists.
\end{lemma}

\begin{proof}
It follows from the definitions of $y_{n}$ and $J^{A}_{\lambda_{n}}$ that 
$w_{n}-\lambda_{n}B(w_{n})\in (I+\lambda_{n}A)(y_{n})$, which implies that
\begin{equation}\label{eq:0.0}
    w_{n}-y_{n}-\lambda_{n}B(w_{n})\in \lambda_{n} A(y_{n}).
\end{equation}

Since $x^\ast\in\Omega$, we have $
-\lambda_{n}B(x^\ast)\in \lambda_{n} A(x^\ast).$
Thus, using the monotonicity of $A$, we get  
\begin{equation}\label{eq:01}
\left\langle y_n - x^\ast,\ w_n - y_n - \lambda_n(B(w_{n}) - B(x^\ast)) \right\rangle \geq 0.
\end{equation}
Furthermore, by the monotonicity of $B$ and the fact that $\lambda_n > 0$, we obtain
\begin{equation}\label{eq:02}
\left\langle y_n - x^\ast,\ \lambda_n\left(B(y_{n}) - B(x^\ast)\right) \right\rangle \geq 0.
\end{equation}
Adding inequalities \eqref{eq:01} and \eqref{eq:02}, we get  
\begin{equation}\label{eq:2.2}
   \left\langle y_n - x^\ast,\ d_{n} \right\rangle 
= \left\langle y_n - x^\ast,\ w_n - y_n - \lambda_n\left(B(w_n) - B(y_{n})\right)\right\rangle 
\geq 0, 
\end{equation}
that is, 
\begin{equation}\label{eq:03}
    \left\langle y_n - x^\ast,\ d_{n} \right\rangle\geq0.
\end{equation}
By Lemma \ref{lemma:inner_product_identity} (i), we get
\begin{align}\label{eq;04}
\| u_n - x^\ast \|^2 &\nonumber= \| (w_n - x^\ast) -\gamma\eta_n d_n \|^2 \\
&= \| w_n - x^\ast \|^2 - 2 \gamma\eta_{n} \left\langle w_n - x^\ast,\ d_n \right\rangle 
+ \gamma^2 \eta_n^2 \| d_n \|^2.
\end{align}
Using \eqref{eq:03} and \eqref{eq;04} and taking into account 
$\eta_n = \frac{\left\langle w_n - y_n,\ d_n \right\rangle}{\| d_n \|^2}$, we obtain
\begin{align}\label{eq:05}
\| u_n - x^\ast \|^2 &\nonumber\leq \| w_n - x^\ast \|^2 - 2\gamma \eta_n 
\left\langle w_n - y_n,\ d_n \right\rangle + \gamma^2 \eta_n^2 \| d_n \|^2\\\nonumber
&= \| w_n - x^\ast \|^2 - 2\gamma\eta_n \left\langle w_n - y_n,\ d_n \right\rangle 
+ \gamma^2 \eta_n \left\langle w_n - y_n,\ d_n \right\rangle\\
&= \| w_n - x^\ast \|^2 - \gamma(2-\gamma) \eta_n \left\langle w_n - y_n,\ d_n \right\rangle.
\end{align}
On the other hand, the definition of $u_{n}$ and $d_{n}$ implies that
\begin{align}\label{eq:06}
\eta_n \left\langle w_n - y_n,\ d_n \right\rangle = \| \eta_n d_n \|^2 
= \frac{1}{\gamma^2} \| u_n - w_n \|^2.
\end{align}
On combining \eqref{eq:05} and \eqref{eq:06}, we get 
\begin{equation}\label{eq:007}
    \| u_n - x^\ast \|^2 \leq \| w_n - x^\ast \|^2 - \frac{(2-\gamma)}{\gamma} \| u_n - w_n \|^2,
\end{equation}
and thus,
\begin{equation}\label{eq:07}
    \| u_n - x^\ast \|^2 \leq \| w_n - x^\ast \|^2.
\end{equation}
From \eqref{defxn} and \eqref{eq:07}, we get

\begin{align}
\| x_{n+1} - x^\ast \|^2
&= \| (1 - \alpha)(z_n - x^\ast) + \alpha (u_n - x^\ast) \|^2 \nonumber\\
&= (1 - \alpha) \| z_n - x^\ast \|^2
+ \alpha \| u_n - x^\ast \|^2
- \alpha (1 - \alpha) \| z_n - u_n \|^2 \label{eq:08a} \\
&\leq (1 - \alpha) \| z_n - x^\ast \|^2
+ \alpha \| w_n - x^\ast \|^2
- \alpha (1 - \alpha) \| z_n - u_n \|^2. \label{eq:08}
\end{align}
From  Lemma \ref{lemma:inner_product_identity} (ii)  and \eqref{defznwnyn}, we get
\begin{align}\label{eq:09}
\|z_n - x^\ast\|^2 &= \|x_n + \beta(x_n - x_{n-1}) - x^\ast\|^2 \nonumber \\
&= \|(1 + \beta)(x_n - x^\ast) - \beta(x_{n-1} - x^\ast)\|^2 \nonumber \\
&= (1 + \beta)\|x_n - x^\ast\|^2 - \beta\|x_{n-1} - x^\ast\|^2 + \beta(1 + \beta)\|x_n - x_{n-1}\|^2.
\end{align}
Similarly, 
\begin{align}\label{eq:10}
\|w_n - x^\ast\|^2 &= \|(1 + \theta_n)(x_n - x^\ast) - \theta_n(x_{n-1} - x^\ast)\|^2 \nonumber \\
&= (1 + \theta_n)\|x_n - x^\ast\|^2 - \theta_n\|x_{n-1} - x^\ast\|^2 
+ \theta_n(1 + \theta_n)\|x_n - x_{n-1}\|^2 .
\end{align}
Since $ x_{n+1} = (1 - \alpha)z_n + \alpha u_n$, we have
\begin{align}\label{eq:11}
\|u_n - z_n\|^2 &= \frac{1}{\alpha^2} \|x_{n+1} - z_n\|^2 \nonumber \\
&= \frac{1}{\alpha^2} \|x_{n+1} - x_n - \beta(x_n - x_{n-1})\|^2 \nonumber \\
&= \frac{1}{\alpha^2} \left[ \|x_{n+1} - x_n\|^2 + \beta^2 \|x_n - x_{n-1}\|^2 
- 2\beta \langle x_{n+1} - x_n, x_n - x_{n-1} \rangle \right] \nonumber \\
&\geq \frac{1 - \beta}{\alpha} \|x_{n+1} - x_n\|^2 
- \frac{\beta(1 - \beta)}{\alpha} \|x_n - x_{n-1}\|^2 .
\end{align}
Substituting \eqref{eq:09}, \eqref{eq:10} and \eqref{eq:11} into \eqref{eq:08}, we get 
\begin{align}\label{eq:12}
\|x_{n+1} - x^\ast\|^2 
&\leq (1 - \alpha)(1 + \beta)\|x_n - x^\ast\|^2 - (1 - \alpha)\beta\|x_{n-1} - x^\ast\|^2 \nonumber \\
&\quad + (1 - \alpha)\beta(1 + \beta)\|x_n - x_{n-1}\|^2 
 + \alpha(1 + \theta_n)\|x_n - x^\ast\|^2 - \alpha \theta_n\|x_{n-1} - x^\ast\|^2 \nonumber \\
&\quad + \alpha \theta_n(1 + \theta_n)\|x_n - x_{n-1}\|^2  - (1 - \alpha)(1 - \beta)
\|x_{n+1} - x_n\|^2 + (1 - \alpha)\beta(1 - \beta)\|x_n - x_{n-1}\|^2 \nonumber \\
&= \left[1 + (1 - \alpha)\beta + \alpha \theta_n\right]\|x_n - x^\ast\|^2 - \left[(1 - \alpha)\beta 
+ \alpha \theta_n\right]\|x_{n-1} - x^\ast\|^2 \nonumber \\
&\quad + \left[2\beta(1 - \alpha) + \alpha \theta_n(1 + \theta_n)\right]\|x_n - x_{n-1}\|^2
- (1 - \alpha)(1 - \beta)\|x_{n+1} - x_n\|^2 .
\end{align}
Define
\begin{align*}
\Delta_n :=& \|x_n - x^\ast\|^2 - \left[(1 - \alpha)\beta + \alpha \theta_n\right] 
\|x_{n-1} - x^\ast\|^2  
+ \left[2\beta(1 - \alpha) + \alpha \theta_n(1 + \theta_n)\right] \|x_n - x_{n-1}\|^2.
\end{align*}
Since  $\alpha > 0$ and \ $\theta_{n} \leq \theta_{n+1}$, 
$\alpha \theta_n \leq \alpha \theta_{n+1}$. Thus,
\begin{align}\label{eq:13}
\Delta_{n+1} - \Delta_n 
&= \|x_{n+1} - x^\ast\|^2 - \left[(1 - \alpha)\beta + \alpha \theta_{n+1}\right] \|x_n - x^\ast\|^2 \nonumber \\
&\quad + \left[2\beta(1 - \alpha) + \alpha \theta_{n+1}(1 + \theta_{n+1})\right] \|x_{n+1} - x_n\|^2 \nonumber \\
&\quad - \|x_n - x^\ast\|^2 + \left[(1 - \alpha)\beta + \alpha \theta_n\right] \|x_{n-1} - x^\ast\|^2 \nonumber \\
&\quad - \left[2\beta(1 - \alpha) + \alpha \theta_n(1 + \theta_n)\right] \|x_n - x_{n-1}\|^2 \nonumber \\
&\leq \|x_{n+1} - x^\ast\|^2 - \left[1 + (1 - \alpha)\beta + \alpha \theta_{n+1}\right] \|x_n - x^\ast\|^2 \nonumber \\
&\quad + \left[(1 - \alpha)\beta + \alpha \theta_n\right] \|x_{n-1} - x^\ast\|^2 \nonumber \\
&\quad - \left[2\beta(1 - \alpha) + \alpha \theta_n(1 + \theta_n)\right] \|x_n - x_{n-1}\|^2 \nonumber \\
&\quad + \left[2\beta(1 - \alpha) + \alpha \theta_{n+1}(1 + \theta_{n+1})\right] \|x_{n+1} - x_n\|^2  \nonumber \\
&\leq \|x_{n+1} - x^\ast\|^2 - \left[1 + (1 - \alpha)\beta + \alpha \theta_{n}\right] \|x_n - x^\ast\|^2 \nonumber \\
&\quad + \left[(1 - \alpha)\beta + \alpha \theta_n\right] \|x_{n-1} - x^\ast\|^2 \nonumber \\
&\quad - \left[2\beta(1 - \alpha) + \alpha \theta_n(1 + \theta_n)\right] \|x_n - x_{n-1}\|^2 \nonumber \\
&\quad + \left[2\beta(1 - \alpha) + \alpha \theta_{n+1}(1 + \theta_{n+1})\right] \|x_{n+1} - x_n\|^2 .
\end{align}
On combining \eqref{eq:12} and \eqref{eq:13}, we get 
\begin{align}\label{eq:14}
\Delta_{n+1} - \Delta_n 
&\leq - (1 - \alpha)(1 - \beta)\|x_{n+1} - x_n\|^2 
+ \left[2\beta(1 - \alpha) + \alpha \theta_{n+1}(1 + \theta_{n+1})\right] \|x_{n+1} - x_n\|^2 \nonumber \\
&= -\left[(1 - \alpha)(1 - \beta) - 2\beta(1 - \alpha) - \alpha \theta_{n+1}(1 + \theta_{n+1})\right] \|x_{n+1} - x_n\|^2 .
\end{align}
From Assumption~\ref{assumption} (A4) and (A5), it follows that:
\begin{align}\label{eq:15}
(1 - \alpha)(1 - \beta) - 2\beta(1 - \alpha) - \alpha \theta_{n+1}(1 + \theta_{n+1}) \nonumber 
&= (1 - \alpha)\left[(1 - \beta) - 2\beta\right] - \alpha \theta_{n+1}(1 + \theta_{n+1}) \nonumber \\
&= (1 - \alpha)(1 - 3\beta) - \alpha \theta_{n+1}(1 + \theta_{n+1}) \nonumber \\
&\geq (1 - \alpha)(1 - 3\beta) - 2\alpha\nonumber\\ &> 0.
\end{align}
Thus, from \eqref{eq:14} and \eqref{eq:15}, we have
\begin{align}\label{eq:15.1}
\Delta_{n+1} - \Delta_n 
&\leq -\left[(1 - \alpha)(1 - 3\beta) - 2\alpha\right] \|x_{n+1} - x_n\|^2, 
\end{align}
and therefore, $\{\Delta_n\}$ is non-increasing.
Similarly, we observe that
\begin{align}\label{eq:16}
\Delta_n &= \|x_n - x^\ast\|^2 - \left[(1 - \alpha)\beta + \alpha \theta_n \right] \|x_{n-1} - x^\ast\|^2 
+ \left[2\beta(1 - \alpha) + \alpha \theta_n(1 + \theta_n)\right] \|x_n - x_{n-1}\|^2 \nonumber \\
&\geq \|x_n - x^\ast\|^2 - \left[(1 - \alpha)\beta + \alpha \theta_n \right] \|x_{n-1} - x^\ast\|^2 .
\end{align}
Thus, rearranging gives
\begin{align}\label{eq:17}
\|x_n - x^\ast\|^2 
&\leq \left[(1 - \alpha)\beta + \alpha \theta_n \right] \|x_{n-1} - x^\ast\|^2 + \Delta_n \nonumber \\
&\leq \left[(1 - \alpha)\beta + \alpha \right] \|x_{n-1} - x^\ast\|^2 + \Delta_n \nonumber \\
&\leq \delta \|x_{n-1} - x^\ast\|^2 + \Delta_1 \nonumber \\
&\leq \delta^2 \|x_{n-2} - x^\ast\|^2 + \delta \Delta_1 + \Delta_1 \nonumber \\
&\ \ \vdots \nonumber \\
&\leq \delta^n \|x_0 - x^\ast\|^2 + \Delta_1 \sum_{k=0}^{n-1} \delta^k \nonumber\\
&\leq \delta^n \|x_0 - x^\ast\|^2 + \frac{\Delta_1}{1 - \delta}, 
\end{align}
where \( \delta := (1 - \alpha)\beta + \alpha \in [0,1) \) by Assumption~\ref{assumption} (A4) and (A5). Hence, $\{\|x_n - x^\ast\|\}$ is bounded and so is $\{x_{n}\}$. Moreover, from the definition of $\Delta_{n}$, we have
\begin{align}
\Delta_{n+1} &= \|x_{n+1} - x^\ast\|^2 - \left[(1 - \alpha)\beta + \alpha \theta_{n+1} \right] \|x_n - x^\ast\|^2 
+ \left[2\beta(1 - \alpha) + \alpha \theta_{n+1}(1 + \theta_{n+1})\right] \|x_{n+1} - x_n\|^2 \nonumber \\
&\geq -\left[(1 - \alpha)\beta + \alpha \theta_{n+1} \right] \|x_n - x^\ast\|^2.\nonumber
\end{align}
By \eqref{eq:17}, it follows that
\begin{align}\label{eq:18}
- \Delta_{n+1} &\leq \left[(1 - \alpha)\beta + \alpha \theta_{n+1} \right] \|x_n - x^\ast\|^2 \nonumber\\&
\leq \left[(1 - \alpha)\beta + \alpha \right] \|x_n - x^\ast\|^2 \nonumber\\&
\leq \delta \|x_n - x^\ast\|^2\nonumber \\&
\leq \delta^{n+1} \|x_0 - x^\ast\|^2 + \frac{\delta \Delta_1}{1 - \delta}. 
\end{align}
From \eqref{eq:15.1} and \eqref{eq:18}, we obtain
\begin{align}\label{eq:19}
\left[(1 - \alpha)(1 - 3\beta) - 2\alpha \right] \sum_{n=1}^k \|x_{n+1} - x_n\|^2 
&\leq \Delta_1 - \Delta_{k+1} \nonumber \\
&\leq \delta^{k+1} \|x_0 - x^\ast\|^2 + \frac{\Delta_1}{1 - \delta},
\end{align}
which implies
\begin{align}\label{eq:19.1}
\sum_{n=1}^\infty \|x_{n+1} - x_n\|^2 
&\leq \frac{\Delta_1}{\left[(1 - \alpha)(1 - 3\beta) - 2\alpha\right](1 - \delta)} < +\infty.
\end{align}
Therefore,
$
\lim\limits_{n \to \infty} \|x_{n+1} - x_n\| = 0.
$
Hence, we obtain from \eqref{defznwnyn} that
\begin{equation}\label{eq:20}
    \lim_{n \to \infty} \|z_n - x_n\| = \lim_{n \to \infty} \|w_n - x_n\| = 0.
\end{equation}
By \eqref{eq:12}, we have
\begin{align}\label{eq:21}
\|x_{n+1} - x^\ast\|^2 
&\leq \left[1 + (1 - \alpha)\beta + \alpha\theta_n\right] \|x_n - x^\ast\|^2 - \left[(1 - \alpha)\beta + \alpha\theta_n\right] \|x_{n-1} - x^\ast\|^2 \nonumber \\
&\quad + \left[2\beta(1 - \alpha) + \alpha\theta_n(1 + \theta_n)\right] \|x_n - x_{n-1}\|^2 \nonumber \\
&\leq \left[1 + (1 - \alpha)\beta + \alpha\theta_n\right] \|x_n - x^\ast\|^2 - \left[(1 - \alpha)\beta + \alpha\theta_n\right] \|x_{n-1} - x^\ast\|^2 \nonumber \\
&\quad + \left(2\beta(1 - \alpha) + 2\alpha\right) \|x_n - x_{n-1}\|^2 \nonumber \\
&= \|x_n - x^\ast\|^2 + \left[(1 - \alpha)\beta + \alpha\theta_n\right]\left(\|x_n - x^\ast\|^2 - \|x_{n-1} - x^\ast\|^2\right) \nonumber \\
&\quad + \left(2\beta(1 - \alpha) + 2\alpha\right) \|x_n - x_{n-1}\|^2. 
\end{align}
Since \( \theta_n \leq 1 < \frac{1 - (1 - \alpha)\beta}{\alpha} \), we have
\(
0 \leq (1 - \alpha)\beta + \alpha\theta_n < 1.
\)
Thus from Lemma \ref{lemma:auxiliaryconvergence} and   \eqref{eq:21}, we get
\begin{equation}\label{eq:22}
\lim_{n \to \infty} \|x_n - x^\ast\| = \ell < {+}\infty. 
\end{equation}
This completes the proof.
\end{proof}

\begin{theorem}\label{th:weakconv}
Let $\{x_{n}\}$  be the sequence generated by Algorithm \ref{algo:IPCMAS1}. 
Suppose that Assumption \ref{assumption} {\rm (A1)-(A5)} hold. 
Then the sequence $\{x_{n}\}$ converges weakly to a point in $\Omega.$
\end{theorem}

\begin{proof}
From \eqref{defznwnyn} and \eqref{defxn}, we have
\begin{align}\label{eq:23}
\|u_n - z_n\| 
&= \frac{1}{\alpha} \|x_{n+1} - x_n - \beta(x_n - x_{n-1})\| \notag\\
&\leq \frac{1}{\alpha} \|x_{n+1} - x_n\| + \frac{\beta}{\alpha} \|x_n - x_{n-1}\|.
\end{align}
Since $\lim\limits_{n\to\infty}\|x_{n+1}-x_{n}\|=0$, we have $\lim\limits_{n\to\infty} \|u_n - z_n\|=0$.
On the other hand, from the triangle inequality, we have
$$
\|u_n - x_n\| \leq \|u_n - z_n\| + \|z_n - x_n\| ,
$$
$$
\|w_n - u_n\| \leq \|w_n - x_n\| + \|x_n - u_n\|.
$$
The above two inequalities together with \eqref{eq:20} implies that
\begin{equation}\label{eq:24}
    \lim_{n \to \infty} \|u_n - x_n\| = 0, \quad \mbox{and} 
    \quad \lim_{n \to \infty} \|w_n - u_n\| = 0.
\end{equation}
From the definition of $d_{n}$ in \eqref{defdn}, we get
\begin{align*}
\|d_n\| &= \|w_n - y_n - \lambda_n(B(w_{n}) - B(y_{n}))\| \\
&\leq \|w_n - y_n\| + \lambda_n \|B(w_{n}) - B(y_{n})\| \\
&\leq \left(1 +\mu\frac{ \lambda_n }{ \lambda_{n+1}}\right) \|w_n - y_n\|,
\end{align*}
which implies that
\begin{equation}\label{eq:25}
    \frac{1}{\|d_n\|} \geq \frac{1}{\left(1 +\mu\frac{ \lambda_n }{ \lambda_{n+1}}\right)~\|w_n - y_n\|}.
\end{equation}
On the other hand, the definition of $d_{n}$ in \eqref{defdn} yields that
\begin{align}\label{eq:26}
\langle w_n - y_n, d_n \rangle 
&= \langle w_n - y_n, w_n - y_n - \lambda_n(B(w_{n}) - B(y_{n})) \rangle \notag\\
&= \|w_n - y_n\|^2 - \langle w_n - y_n, \lambda_n(B(w_{n}) - B(y_{n})) \rangle \notag\\
&\geq \|w_n - y_n\|^2 - \lambda_n \|B(w_{n}) - B(y_{n})\| \cdot \|w_n - y_n\| \notag\\
&\geq \|w_n - y_n\|^2 - \mu\frac{\lambda_n }{\lambda_{n+1}} \|w_n - y_n\|^2 \notag\\
&= \left(1 -\mu \frac{\lambda_n }{\lambda_{n+1}}\right) \|w_n - y_n\|^2.
\end{align}
From the definitions of $u_n$ and $\eta_n$, together with \eqref{eq:25} and \eqref{eq:26}, 
we obtain
\begin{align}\label{eq:27}
\|u_n - w_n\| 
&= \gamma \eta_n \|d_n\| = \gamma \frac{\langle w_n - y_n, d_n \rangle}{\|d_n\|} \notag\\
&\geq \gamma\,
\frac{\left(1 - \mu\frac{\lambda_n }{\lambda_{n+1}}\right)}{
\left(1 + \mu\frac{\lambda_n}{\lambda_{n+1}}\right)} \|w_n - y_n\|.
\end{align}
Thus, by using \eqref{eq:24}, \eqref{eq:27}, and taking into account that
$\lim_{n\to\infty}\limits\lambda_n = \lambda$ and ${\mu\in (0,1)}$ we obtain
\[
\lim_{n\to\infty}\|w_n - y_n\| = 0.
\]
Hence,
\[
\|x_n - y_n\| 
\le \|w_n - y_n\| + \|w_n - x_n\|
\rightarrow 0,
\qquad \text{as } n \to \infty.
\]

Since the sequence $\{x_n\}$ is bounded, there exists a subsequence 
$\{x_{n_k}\}$ and a point $x^\ast \in \mathscr{H}$ such that 
$x_{n_k} \rightharpoonup x^\ast$.
Moreover, since $\|x_n - y_n\| \to 0$, we also have $y_{n_k} \rightharpoonup x^\ast.$
{Note that the sequence $\{x_{n_k} - y_{n_k}\}$ is bounded, since both $\{x_{n_k}\}$ and $\{y_{n_k}\}$ are bounded.}

Let $(x,y) \in \operatorname{Graph}(A+B)$, that is, $y \in (A+B)(x)$.  
Then, $y-B(x)\in A(x).$  Since  
$
y_{n_{k}} = J^{A}_{\lambda_{n_{k}}}\bigl(w_{n_{k}} - \lambda_{n_{k}} B(w_{n_{k}})\bigr)
$, we have 
\[
\frac{1}{\lambda_{n_{k}}}\bigl(w_{n_{k}} - \lambda_{n_{k}} B(w_{n_{k}})-y_{n_{k}}\bigr)\in A(y_{n_{k}}).
\]
Thus, from the monotonicity of $A$, we have
\[
\left\langle y- B(x)-
\frac{1}{\lambda_{n_k}} \bigl(w_{n_{k}} - \lambda_{n_{k}} B(w_{n_{k}})-y_{n_{k}}\bigr),\,x-y_{n_{k}}\right\rangle \ge 0,
\]
and so,
\begin{align*}
\langle y,\, x - y_{n_k} \rangle 
&\ge 
\left\langle 
B(x) + \frac{1}{\lambda_{n_k}}
\bigl( w_{n_k} - \lambda_{n_k} B(w_{n_k}) - y_{n_k} \bigr),\,
x - y_{n_k}
\right\rangle \\[6pt]
&=
\frac{1}{\lambda_{n_k}} \langle w_{n_k} - y_{n_k},\, x - y_{n_k} \rangle
+ \langle B(x) - B(y_{n_k}),\, x - y_{n_k} \rangle
+ \langle B(y_{n_k}) - B(w_{n_k}),\, x - y_{n_k} \rangle \\[6pt]
&\ge
\frac{1}{\lambda_{n_k}} \langle w_{n_k} - y_{n_k},\, x - y_{n_k} \rangle
+ \langle B(y_{n_k}) - B(w_{n_k}),\, x - y_{n_k} \rangle .
\end{align*}
Passing the limit in the last inequality and noting  that $\|w_{n_k} - y_{n_k}\|\to 0$ 
and $y_{n_{k}}\rightharpoonup x^\ast$, and by the Lipschitz continuity of $B$, we obtain
$ \langle {x-x^\ast ,\, y}\rangle \geq 0$ for all $(x,y) \in  \operatorname{Graph}(A+B)$. 
By using the maximal monotonicity of $A+B$, we get \( x^\ast\in \Omega \). {It follows from Lemma~\ref{lemma:boundedness} that $\lim\limits_{n \to \infty} \|x_n - x^\ast\|$ exists. Hence, by Lemma~\ref{opialcondition}, we conclude that the whole sequence $\{x_n\}$ converges weakly to a point in $\Omega$.}
This completes the proof.
\end{proof}

{ The following theorem provides a nonasymptotic convergence rate for Algorithm~\ref{algo:IPCMAS1} for solving VIP~\eqref{mainproblem}.
\begin{theorem}\label{thm:rate}
Let $\{x_n\}$ be a sequence generated by Algorithm~\ref{algo:IPCMAS1}. Suppose that Assumption \ref{assumption} {\rm (A1)-(A5)} hold. Then there exist a positive integer $N$ and a constant $C>0$ such that
\[
\min_{N\le j\le n}\|y_j-w_j\|
\le \left(\frac{C}{n-N+1}\right)^{1/2},\qquad \forall n\ge N.
\]
Consequently,
\[
\min_{N\le j\le n}\|y_j-w_j\|
=\mathcal{O}\left((n-N+1)^{-1/2}\right).
\]
\end{theorem}
\begin{proof}
From \eqref{eq:007}, \eqref{eq:08a}, \eqref{eq:09}, and \eqref{eq:27}, and using the fact that $\theta_n \leq 1$, we obtain
\begin{align}\label{eq:B5} 
\|x_{n+1} - x^\ast\|^2
&\le \left[1 + (1 - \alpha)\beta + \alpha\, \theta_n\right]\|x_n - x^\ast\|^2 - \left[(1 - \alpha)\beta + \alpha\, \theta_n\right]\|x_{n-1} - x^\ast\|^2 \nonumber\\ 
&\quad+ \left[2\beta(1 - \alpha) + \alpha\, \theta_n(1 + \theta_n)\right]\|x_n - x_{n-1}\|^2 \nonumber\\
&\quad- (1 - \alpha)(1 - \beta)\|x_{n+1} - x_n\|^2 - \alpha\,\gamma\,(2-\gamma)\,\frac{\left(1 - \mu \frac{\lambda_n}{\lambda_{n+1}}\right)^2}{\left(1 + \mu \frac{\lambda_n}{\lambda_{n+1}}\right)^2}\|y_n-w_n\|^2\nonumber\\ 
&\le\|x_n - x^\ast\|^2 +\left((1 - \alpha)\beta + \alpha \theta_n\right)\left[\|x_{n} - x^\ast\|^2-\|x_{n-1} - x^\ast\|^2\right] \nonumber\\ 
&\quad+ \left[2\beta(1 - \alpha) + \alpha\, \theta_n(1 + \theta_n)\right]\|x_n - x_{n-1}\|^2 - \alpha\,\gamma\,(2-\gamma)\,\frac{\left(1 - \mu \frac{\lambda_n}{\lambda_{n+1}}\right)^2}{\left(1 + \mu \frac{\lambda_n}{\lambda_{n+1}}\right)^2}\|y_n-w_n\|^2\nonumber\\ 
&\le\|x_n - x^\ast\|^2 +\left((1 - \alpha)\beta + \alpha\right)\left[\|x_{n} - x^\ast\|^2-\|x_{n-1} - x^\ast\|^2\right] \nonumber\\ 
&\quad+ \left[2\beta(1 - \alpha) + 2\,\alpha \right]\|x_n - x_{n-1}\|^2 - \alpha\,\gamma\,(2-\gamma)\,\frac{\left(1 - \mu \frac{\lambda_n}{\lambda_{n+1}}\right)^2}{\left(1 + \mu \frac{\lambda_n}{\lambda_{n+1}}\right)^2}\,\|y_n-w_n\|^2\nonumber\\ 
&\leq\|x_n - x^\ast\|^2 +a\left[\|x_{n} - x^\ast\|^2-\|x_{n-1} - x^\ast\|^2\right] \nonumber\\
&\quad+ b\, \|x_n - x_{n-1}\|^2 - \alpha\,\gamma\,(2-\gamma)\frac{\left(1 - \mu \frac{\lambda_n}{\lambda_{n+1}}\right)^2}{\left(1 + \mu \frac{\lambda_n}{\lambda_{n+1}}\right)^2}\,\|u_n-w_n\|^2,
\end{align} 
where $a=(1 - \alpha)\beta + \alpha$ and $b=2\beta(1 - \alpha) + 2\,\alpha.$  Since $\lim_{n \to \infty} \lambda_n = \lambda$, it follows that
\[
\lim_{n \to \infty} 
\frac{\left(1 - \mu \frac{\lambda_n}{\lambda_{n+1}}\right)^2}
{\left(1 + \mu \frac{\lambda_n}{\lambda_{n+1}}\right)^2}
= \frac{(1-\mu)^2}{(1+\mu)^2}.
\]
Therefore, since $\dfrac{(1-\mu)^2}{(1+\mu)^2}>0$, there exists an integer $N \geq 1$ such that
\begin{equation}\label{eq:LB}
\frac{\left(1 - \mu \frac{\lambda_n}{\lambda_{n+1}}\right)^2}
{\left(1 + \mu \frac{\lambda_n}{\lambda_{n+1}}\right)^2}
\geq \frac{1}{2}\frac{(1-\mu)^2}{(1+\mu)^2}, 
\quad \text{for all } n \geq N.
\end{equation}

Combining \eqref{eq:B5} and \eqref{eq:LB}, and rearranging terms, we obtain for all $n\ge N$,
\begin{align}\label{eq:B6}
\frac{1}{2}\alpha\,\gamma\,(2-\gamma)\frac{(1-\mu)^2}{(1+\mu)^2}\|y_n-w_n\|^2
&\le \|x_n-x^\ast\|^2-\|x_{n+1}-x^\ast\|^2 \nonumber\\
&\quad +a\left[\|x_n-x^\ast\|^2-\|x_{n-1}-x^\ast\|^2\right] \nonumber\\
&\quad +b\|x_n-x_{n-1}\|^2.
\end{align}

Set
$
\varphi_n:=\|x_n-x^\ast\|^2,~
\Gamma_n:=\varphi_n-\varphi_{n-1},~
\varsigma_n:=b\|x_n-x_{n-1}\|^2,
$ and 
$\zeta:=\frac{1}{2}\alpha\,\gamma\,(2-\gamma)\frac{(1-\mu)^2}{(1+\mu)^2}.$
Then \eqref{eq:B6} can be written as
\begin{align}\label{eq:B10}
\zeta\|y_n-w_n\|^2
&\le \varphi_n-\varphi_{n+1}+a\Gamma_n+\varsigma_n \nonumber\\
&\le \varphi_n-\varphi_{n+1}+a[\Gamma_n]_+ +\varsigma_n,
\qquad \forall n\ge N,
\end{align}
where $[\Gamma_n]_+:=\max\{\Gamma_n,0\}$.

On the other hand, it follows from \eqref{eq:B5} that
$\Gamma_{n+1}\le a\Gamma_n+\varsigma_n,$\quad for all $n\ge N$,

Hence,
\begin{equation}\label{eq:B15}
[\Gamma_{n+1}]_+\le a[\Gamma_n]_+ +\varsigma_n,
\qquad \forall n\ge N,
\end{equation}
which implies
\begin{align}\label{eq:B16}
[\Gamma_{n+1}]_+ 
&\le a\, [\Gamma_{n}]_+ + \varsigma_n \nonumber\\
&\le a^2\, [\Gamma_{n-1}]_+ + a\, \varsigma_{n-1} + \varsigma_n \nonumber\\
&\;\;\vdots \nonumber\\
&\le a^{n-N+1}\, [\Gamma_{N}]_+ + a^{n-N}\, \varsigma_{N} + \cdots + a\,\varsigma_{n-1} + \varsigma_n \nonumber\\
&= a^{n-N+1} [\Gamma_{N}]_+ + \sum_{j=1}^{n-N+1} a^{j-1} \varsigma_{n+1-j}.
\end{align}
Thus,
\begin{align}\label{eq:B18}
\sum_{n=N}^{\infty}[\Gamma_{n+1}]_+
&\le \frac{a}{1-a}[\Gamma_N]_+ +\frac{1}{1-a}\sum_{n=N}^{\infty}\varsigma_n.
\end{align}

In view of \eqref{eq:19.1}, we have
\(
\sum\limits_{n=1}^{\infty}\|x_n-x_{n-1}\|^2<\infty.
\)
Thus, there exists $M>0$ such that
\[
\sum_{n=1}^{\infty}\|x_n-x_{n-1}\|^2\le M.
\]
Using the definition of $\varsigma_n$, we get
\begin{equation}\label{eq:B21}
\sum_{n=1}^{\infty}\varsigma_n
=
b\sum_{n=1}^{\infty}\|x_n-x_{n-1}\|^2
\le b\,M
=:M_1.
\end{equation}
Combining \eqref{eq:B18} and \eqref{eq:B21}, we obtain
\begin{equation}\label{eq:B22}
\sum_{n=N}^{\infty}[\Gamma_{n+1}]_+
\le \frac{a}{1-a}[\Gamma_N]_+ +\frac{M_1}{1-a}.
\end{equation}

Summing \eqref{eq:B10} from $j=N$ to $n$, we arrive at
\begin{align}\label{eq:B23}
\zeta\sum_{j=N}^{n}\|y_j-w_j\|^2
&\le \sum_{j=N}^{n}(\varphi_j-\varphi_{j+1})
+a\sum_{j=N}^{n}[\Gamma_j]_+
+\sum_{j=N}^{n}\varsigma_j \nonumber\\
&\le \varphi_N
+a\left([\Gamma_N]_+ +\sum_{j=N}^{\infty}[\Gamma_{j+1}]_+\right)
+\sum_{j=N}^{\infty}\varsigma_j \nonumber\\
&\le \varphi_N+\frac{a}{1-a}[\Gamma_N]_+ +\frac{M_1}{1-a}.
\end{align}
Consequently,
\begin{equation}\label{eq:B24}
\sum_{j=N}^{n}\|y_j-w_j\|^2
\le
\left[
\|x_N-x^\ast\|^2
+\frac{a}{1-a}\left[\|x_N-x^\ast\|^2-\|x_{N-1}-x^\ast\|^2\right]_+
+\frac{M_1}{1-a}
\right]\frac{1}{\zeta}.
\end{equation}

Hence,
\begin{align*}
\min_{N\le j\le n}\|y_j-w_j\|
\le
\left[
\frac{1}{n-N+1}
\left(
\|x_N-x^\ast\|^2
+\frac{a}{1-a}\left[\|x_N-x^\ast\|^2-\|x_{N-1}-x^\ast\|^2\right]_+
+\frac{M_1}{1-a}
\right)\frac{1}{\zeta}
\right]^{1/2}.
\end{align*}
Define
\[
C:=
\left(
\|x_N-x^\ast\|^2
+\frac{a}{1-a}\left[\|x_N-x^\ast\|^2-\|x_{N-1}-x^\ast\|^2\right]_+
+\frac{M_1}{1-a}
\right)\frac{1}{\zeta}.
\]
Then
\[
\min_{N\le j\le n}\|y_j-w_j\|
\le \left(\frac{C}{n-N+1}\right)^{1/2},\qquad \forall n\ge N.
\]
This completes the proof.
\end{proof}

\begin{remark}

It follows from Remark~\ref{rm:wnyn} that if $y_n = w_n$, then $y_n \in \Omega$. Hence, if $\|y_n - w_n\|< \varepsilon$, then $y_n$ is regarded as an $\varepsilon$-approximate solution of \eqref{mainproblem}.In view of Theorem~\ref{thm:rate}, to obtain such an iterate, that is an iterate satisfying $\|y_n - w_n\|<\varepsilon$, Algorithm~\ref{algo:IPCMAS1} requires at most
\[
\left\lceil
\frac{
\|x_N-x^\ast\|^2
+\dfrac{a}{1-a}\left[\|x_N-x^\ast\|^2-\|x_{N-1}-x^\ast\|^2\right]_+
+\dfrac{M_1}{1-a}
}{
\zeta\,\varepsilon^2
}
\right\rceil + N - 1
\]
iterations. Therefore, the residual iteration complexity is of order $\mathcal{O}(\varepsilon^{-2})$.
\end{remark}
}

\subsection{Strong Convergence}

In this subsection, we establish the strong convergence of Algorithm \ref{algo:IPCMAS1} 
under the assumption of {$\nu$-strong monotonicity} of $A$.

If $A$ is {$\nu$-strongly monotone},
then the problem \eqref{mainproblem} has a unique solution (see, for example, \cite{bauschke_convex_2017}).

\begin{theorem}\label{thm:strong_convergence}
Suppose that Assumption \ref{assumption} {\rm (A1)-(A4)} hold.
If $A$ is {$\nu$-strongly monotone},
then the sequence $\{x_n\}$ generated by Algorithm \ref{algo:IPCMAS1} converges strongly
to a unique solution in $\Omega$.
\end{theorem}

\begin{proof}
Let $x^\ast\in \Omega$ be the unique solution of the problem \eqref{mainproblem}. 
Then, we have {$-\lambda_{n}B(x^\ast)\in \lambda_{n}A(x^\ast)$}.
From the $\nu$-strong monotonicity of $A$ and \eqref{eq:0.0}, we have
\begin{equation}\label{eq:0.01}
    \left\langle w_{n} - y_{n} - \lambda_{n} B w_{n} + \lambda_{n} B (x^\ast),\ y_{n} - x^\ast\right\rangle
\geq \nu \lambda_{n} \| y_{n} - x^\ast \|^{2}.
\end{equation}
By \eqref{eq:2.2} and \eqref{eq:0.01}, we get
$$\left\langle y_{n}-x^\ast, d_{n}\right\rangle=\left\langle\ y_{n} - x^\ast,\, w_{n} - y_{n} - \lambda_{n} B w_{n} + \lambda_{n} B (x^\ast)\right\rangle
\geq \nu \lambda_{n} \| y_{n} - x^\ast \|^{2}.$$
That is,
\begin{equation}\label{eq:sc1}
    \left\langle y_{n}-x^\ast, d_{n}\right\rangle\geq \nu \lambda_{n} \| y_{n} - x^\ast \|^{2}.
\end{equation}
Since
\begin{equation}\label{eq:sc0.0}
\langle w_n - x^\ast, d_n \rangle
= \langle w_n - y_n, d_n \rangle
+ \langle y_n - x^\ast, d_n \rangle,  
\end{equation}
from \eqref{eq:sc1} and \eqref{eq:sc0.0}, we obtain
\begin{equation}\label{eq:sc2.0}
\langle w_n - x^\ast, d_n \rangle
\ge
\langle w_n - y_n, d_n \rangle
+ \nu \lambda_n \|y_n - x^\ast\|^2 .
\end{equation}
Substituting \eqref{eq:sc2.0} into \eqref{eq;04}, we have
\begin{align}\label{eq:sc2}
\|u_n - x^\ast\|^2
\le\;&
\|w_n - x^\ast\|^2
+ \gamma^2 \eta_n^2 \|d_n\|^2
- 2\gamma \eta_n \langle w_n - y_n, d_n \rangle
- 2\gamma \eta_n \nu \lambda_n \|y_n - x^\ast\|^2 \nonumber \\
=\;&
\|w_n - x^\ast\|^2
+ \gamma^2 \eta_n \langle w_n - y_n, d_n \rangle
- 2\gamma \eta_n \langle w_n - y_n, d_n \rangle
- 2\gamma \eta_n \nu \lambda_n \|y_n - x^\ast\|^2 \nonumber \\
=\;&
\|w_n - x^\ast\|^2
- \gamma(2-\gamma)\eta_n
\langle w_n - y_n, d_n \rangle
- 2\gamma \eta_n \nu \lambda_n \|y_n - x^\ast\|^2 .
\end{align}
From \eqref{eq:06} and \eqref{eq:sc2}, we obtain
\begin{equation}\label{eq:sc2.2}
  \|u_n - x^\ast\|^2
\le
\|w_n - x^\ast\|^2
- \frac{(2-\gamma)}{\gamma} \|u_n - w_n\|^2
- 2\gamma\,{\eta_n}\,\lambda_n \|y_n - x^\ast\|^2 .
\end{equation}
Thus,
\begin{equation}\label{eq:sc03}
    \|u_n - x^\ast\|^2
\le
\|w_n - x^\ast\|^2-
2\gamma\,{\eta_n}\,\lambda_n \|y_n - x^\ast\|^2 .
\end{equation}
On combining \eqref{eq:08a} and \eqref{eq:sc03}, we get
\begin{align}\label{eq:sc04}
\|x_{n+1}-x^\ast\|^2&\le (1-\alpha)\|z_n-x^\ast\|^2
+ \alpha\|w_n-x^\ast\|^2
- \alpha(1-\alpha)\|z_n-u_n\|^2\nonumber\\
&\quad- 2\alpha\gamma\eta_n\nu\lambda_n\|y_n-x^\ast\|^2  . 
\end{align}
Substituting \eqref{eq:09}, \eqref{eq:10}, and \eqref{eq:11} into \eqref{eq:sc04}, we obtain
\begin{align}
\|x_{n+1}-x^\ast\|^2
\le\;&
(1-\alpha)(1+\beta)\|x_n-x^\ast\|^2
- (1-\alpha)\beta\|x_{n-1}-x^\ast\|^2 \nonumber\\
&+ (1-\alpha)\beta(1+\beta)\|x_n-x_{n-1}\|^2
+ \alpha(1+\theta_n)\|x_n-x^\ast\|^2 \nonumber\\
&- \alpha\theta_n\|x_{n-1}-x^\ast\|^2
+ \alpha\theta_n(1+\theta_n)\|x_n-x_{n-1}\|^2 \nonumber\\
&- (1-\alpha)(1-\beta)\|x_{n+1}-x_n\|^2
+ (1-\alpha)\beta(1-\beta)\|x_n-x_{n-1}\|^2 \nonumber\\
&- 2\alpha\gamma\eta_n\nu\lambda_n\|y_n-x^\ast\|^2 \nonumber\\
=\;&
\bigl[1+(1-\alpha)\beta+\alpha\theta_n\bigr]\|x_n-x^\ast\|^2
- \bigl[(1-\alpha)\beta+\alpha\theta_n\bigr]\|x_{n-1}-x^\ast\|^2 \nonumber\\
&+ \bigl[2\beta(1-\alpha)+\alpha\theta_n(1+\theta_n)\bigr]\|x_n-x_{n-1}\|^2 \nonumber\\
&- (1-\alpha)(1-\beta)\|x_{n+1}-x_n\|^2
- 2\alpha\gamma\eta_n\nu\lambda_n\|y_n-x^\ast\|^2 .
\end{align}
Define 
\[
\vartheta := \sup_{n\ge 1}\Bigl\{2\beta(1-\alpha)+\alpha\theta_n(1+\theta_n)\Bigr\}.
\]
Therefore,
\begin{align}\label{eq:sc4.1}
2\alpha\gamma\nu\,\eta_n\lambda_n \|y_n-x^\ast\|^2
&\le \|x_n-x^\ast\|^2 - \|x_{n+1}-x^\ast\|^2
+ \vartheta \|x_n-x_{n-1}\|^2\nonumber\\
&+ \bigl[(1-\alpha)\beta + \alpha\theta_n\bigr]
\bigl[\|x_n-x^\ast\|^2 - \|x_{n-1}-x^\ast\|^2\bigr].
\end{align}
{Since $\mu \in (0,1)$,} fix any constant $\lambda^\ast \in (\mu,1)$.  Since $\lambda_n \to \lambda$ as $n \to \infty$,
\(
\lim\limits_{n \to \infty} \frac{\lambda_n \mu}{\lambda_{n+1}} = \mu < \lambda^\ast.
\)
Hence, there exists an integer $N> 0$ such that
\(
\frac{\lambda_n \mu}{\lambda_{n+1}} < \lambda^\ast,
\quad \forall n \ge N.
\)
{
Consider the function $f(t) = \frac{1-t}{(1+t)^2}$, which is strictly decreasing on $(0,1)$ since $f'(t) = \frac{t-3}{(1+t)^3} < 0$ for $t \in (0,1)$. Since $\frac{\lambda_n \mu}{\lambda_{n+1}} < \lambda^\ast < 1$ for all $n \ge N$, we obtain
\begin{equation}\label{eq:sc07}
\frac{1 - \lambda^\ast}{(1 + \lambda^\ast)^2}
<
\frac{1 - \frac{\lambda_n \mu}{\lambda_{n+1}}}
{\left( 1 + \frac{\lambda_n \mu}{\lambda_{n+1}} \right)^2},
\quad \forall n \ge N.
\end{equation}
}
Now, using relations \eqref{eq:25} and \eqref{eq:26} together with inequality \eqref{eq:sc07}, we deduce that
\begin{align}\label{eq:sc08}
\eta_n
&= \frac{\langle w_n - y_n , d_n \rangle}{\| d_n \|^2} \nonumber \\
&\ge
\frac{\left(1 - \frac{\lambda_n \mu}{\lambda_{n+1}}\right)\| w_n - y_n \|^2}
{\left(1 + \frac{\lambda_n \mu}{\lambda_{n+1}}\right)^2 \| w_n - y_n \|^2} \nonumber \\
&=
\frac{1 - \frac{\lambda_n \mu}{\lambda_{n+1}}}
{\left(1 + \frac{\lambda_n \mu}{\lambda_{n+1}}\right)^2} \nonumber \\
&\ge
\frac{1 - \lambda^\ast}{(1 + \lambda^\ast)^2},
\quad \forall n \ge N.
\end{align}
Substituting inequality \eqref{eq:sc08} into \eqref{eq:sc4.1}, we obtain
\begin{align}
2 \alpha \gamma
\frac{1 - \lambda^\ast}{(1 + \lambda^\ast)^2}
\lambda_n \| y_n - x^\ast \|^2
&\le
2 \alpha \gamma \eta_n \lambda_n \| y_n - x^\ast \|^2 \nonumber \\
&\le
\big[(1 - \alpha)\beta + \alpha \theta_n\big]
\big[ \| x_n - x^\ast \|^2 - \| x_{n-1} - x^\ast \|^2 \big] \nonumber \\
&\quad
+ \| x_n - x^\ast \|^2 - \| x_{n+1} - x^\ast \|^2
+ \vartheta \| x_n - x_{n-1} \|^2 .
\end{align}
Therefore,
\begin{equation}\label{eq:59}
\begin{aligned}
{2 \alpha \gamma
\frac{1 - \lambda^\ast}{(1 + \lambda^\ast)^2}
\min\left\{ \lambda_1, \frac{\mu}{L} \right\}
\sum_{k=N}^{n} \|y_k-x^\ast\|^2}
\le\;&
\|x_{N}-x^\ast\|^2 - \|x_{n+1}-x^\ast\|^2 \\
&+ \vartheta \sum_{k=N}^{n} \|x_k-x_{k-1}\|^2
+ \bigl[(1-\alpha)\beta+\alpha\theta_n\bigr]\|x_n-x^\ast\|^2 \\
&- \bigl[(1-\alpha)\beta+\alpha\theta_{N-1}\bigr]\|x_{N-1}-x^\ast\|^2 .
\end{aligned}
\end{equation}
Since the sequence $\{x_n\}$ is bounded and
\(
\sum_{k=N}^{\infty}\|x_k-x_{k-1}\|^2 < {+}\infty
\quad\text{by} \eqref{eq:19.1},
\)
it follows from \eqref{eq:59} that
\(
\sum_{k=N}^{\infty}\|y_k-x^\ast\|^2 < {+}\infty.
\)
Hence
\[
\lim_{n\to\infty}\|y_n-x^\ast\| = 0.
\]
Consequently, we get
\begin{equation}\label{eq:60}
\|x_n-x^\ast\|
\le \|x_n-w_n\| + \|w_n-y_n\| + \|y_n-x^\ast\|
\longrightarrow 0
\quad\text{as } n\to\infty .
\end{equation}
This completes the proof.
\end{proof}


\subsection{Linear Convergence}\label{S:4}


Recall that a sequence $\{x_n\}$ in a Hilbert space $\mathscr{H}$ is said to converge 
\emph{R-linearly} to $x^\ast$ with rate $\rho \in [0,1)$ {\rm\cite{bauschke_convex_2017}}
if there exists a constant $c > 0$ such that
\[
\|x_n - x^\ast\| \le c\, \rho^{n}, \quad \forall n \in \mathbb{N}.
\]

{In this subsection, we analyze the linear convergence of a parameter-restricted variant of Algorithm~\ref{algo:IPCMAS1}, presented below as Algorithm~\ref{algo:IPCMAS2} (\algoT\!\!).} To obtain linear convergence, we modify the algorithm parameters under the assumptions stated below.
\begin{assumption}\label{linearassm}
\begin{itemize}
    \item[\textnormal{{(A1$'$)}}]  $A\colon\mathscr{H}\to 2^{\mathscr{H}}$ is {$\nu$-strongly monotone and maximal monotone}.

    \item[\textnormal{{(A2$'$)}}]  $B\colon\mathscr{H}\to \mathscr{H}$ is $L$-Lipschitz continuous and monotone.

    \item[\textnormal{{(A3$'$)}}] $\beta =0$ and $0 \leq \theta_n=\theta$,
    $0\leq\theta<\frac{1}{\tau}-1$, where {$\tau:=1-\frac{1}{4}\min\left\{
\frac{(2\gamma - \gamma^2)(1 - \mu)^2}{(1 + \mu)^2}, \;
2\,\nu \gamma\frac{1 - \mu}{(1 + \mu)^2}\min\left\{ \lambda_1, \frac{\mu}{L} \right\}
\right\}.$}
\item[\textnormal{{(A4$'$)}}]  $0 < \alpha < \frac{1}{3}$.
\end{itemize}
\end{assumption}

Hence Algorithm \ref{algo:IPCMAS1} reduces to the following under Assumption \ref{linearassm}.
\begin{algorithm}[H]
	\begin{footnotesize}
		\begin{description}
			\item[Step 0.] Let $\gamma\in (0,2)$, $\mu\in(0,1)$ and $\lambda_{1}>0$. Choose $ x_{0}, x_{1}\in\mathscr{H}$ be given starting points. Set $n:=1$
			\item[Step 1.]  Compute

\begin{equation}\label{linearwnyn}
 \begin{cases}
w_{n}=x_{n}+\theta\left(x_{n}-x_{n-1} \right)\\
y_n = J^{A}_{\lambda_{n}}\left(w_n - \lambda_n B(w_{n})\right).
\end{cases}
\end {equation}
 If $y_{n}=w_{n}$, then stop $y_{n}\in \Omega$.  Otherwise go to \textbf{Step~2}.
			\item[Step 2.] Compute $u_{n}=w_{n}-\gamma\eta_{n}d_{n}$, where
                        \begin{equation}\label{lineardn}
            d_{n}=w_{n}-y_{n}-\lambda_{n}\left(B(w_{n})-B(y_{n})\right),
            \end{equation}
        \textit{and}
        \begin{equation}\label{linearetan}
            \eta_{n}=
	\begin{cases}
		\frac{\left\langle w_{n}-y_{n},d_{n}\right\rangle}{\|d_{n}\|^2},& \mbox{if}~ d_{n}\ne0,\\[5pt]
		0 &\mbox{if}~ d_{n}=0.
	\end{cases}
        \end{equation}

                \item[Step 3.] Compute
                \begin{equation}\label{lineardefxn}
                x_{n+1}=(1-\alpha)x_{n}+\alpha u_{n}.
                \end{equation}
           Update     \begin{equation}\label{lineardeflambda}
	\lambda_{n+1}=
	\begin{cases}
		\min\left\{\frac{\mu \|w_{n}-y_{n}\|}{\|B(w_{n})-B(y_{n})\|},
		\lambda_{n}+\xi_{n}\right\} & \mbox{if } B(w_{n})\neq B(y_{n}),\\[5pt]
		\lambda_{n}+\xi_{n} & \mbox{otherwise}.
	\end{cases}
\end{equation}
                Set $n:=n+1$ and go to  \textbf{Step~1}.
		\end{description}
		
		\caption{New Inertial Projection and Contraction Method with Adaptive Step-size (\algoT\!\!)}
        \label{algo:IPCMAS2}
	\end{footnotesize}
\end{algorithm}

{
\begin{remark}
Algorithm~\ref{algo:IPCMAS2} is a special case of Algorithm~\ref{algo:IPCMAS1} obtained by setting $\beta = 0$ in assumption~(A3$'$).
\end{remark}
}

\begin{theorem}\label{thm:rlinear}
Let $\{x_{n}\}$ be a sequence generated by Algorithm {\rm\ref{algo:IPCMAS2}},
and suppose that the Assumption {\rm\ref{linearassm}} {{\rm (A1$'$)-(A4$'$)}} hold.
Then, the sequence $\{x_{n}\}$ converges strongly to a unique point of $\Omega$
with R-linear convergence rate.
\end{theorem}

\begin{proof}
Let $x^\ast\in \Omega$ be the unique solution of problem \eqref{mainproblem}. 
Then, we have $-\lambda_{n}B(x^\ast)\in \lambda_{n}A(x^\ast)$.  By repeating the same arguments used in \eqref{eq:sc0.0}–\eqref{eq:sc2.2} we get

\begin{equation}\label{eq:32.1}
\|u_{n}-x^\ast\|^2\leq \|w_{n}-x^\ast\|^{2}- \frac{(2-\gamma)}{\gamma}\, \| u_n - w_n \|^2 -2\gamma\, \nu\,\eta_{n}\,\lambda_{n}\|y_{n}-x^\ast\|^{2}.
\end{equation}
On the other hand, from the definition of $\eta_{n}$ in {\eqref{linearetan}}, \eqref{eq:25} and \eqref{eq:26}, we have
\begin{equation}\label{eq:33}
   \eta_{n} = \frac{\langle w_{n} - y_{n},\ d_{n} \rangle}{\| d_{n} \|^{2}} 
\geq\frac{\left( 1 - \mu \frac{\lambda_{n}}{\lambda_{n+1}} \right)}{\left( 1 + \mu \frac{\lambda_{n}}{\lambda_{n+1}} \right)^{2}}.
\end{equation}
By using \eqref{eq:27} and \eqref{eq:33} in \eqref{eq:32.1}, we get

\begin{align*}
\| u_{n} - x^\ast \|^{2} 
\leq \| w_{n} - x^\ast \|^{2} 
- \gamma(2-\gamma)\,
\frac{\left(1- \mu\frac{\lambda_{n}}{\lambda_{n+1}} \right)^{2} }{
\left( 1 + \mu\frac{\lambda_{n}}{\lambda_{n+1}} \right) ^{2}} 
\| w_{n} - y_{n} \|^{2}
- 2\, \nu\, \gamma\, \lambda_{n} \frac{\left(1- \frac{ \mu\lambda_{n}}{\lambda_{n+1}} \right) }{
\left( 1 + \frac{\mu\lambda_{n}}{\lambda_{n+1}} \right) ^{2}} 
 \| y_{n} - x^\ast \|^{2}.
\end{align*}
From Lemma \ref{lemma:Lipschitz_lambda_convergence}, we have
$$
\lim\limits_{n \to \infty} 
\gamma(2-\gamma) 
\frac{\left(1- \mu\frac{ \lambda_{n}}{\lambda_{n+1}} \right)^{2} }{
\left( 1 + \mu\frac{\lambda_{n}}{\lambda_{n+1}} \right) ^{2}} 
= (2\gamma - \gamma^2)\frac{(1 - \mu)^2}{(1 + \mu)^2}$$
and
{
$$
\lim\limits_{n \to \infty} 2\,\nu\, \gamma\,
  \lambda_{n} \frac{\left(1- \mu\frac{\lambda_{n}}{\lambda_{n+1}} \right) }{
\left( 1 + \mu\frac{\lambda_{n}}{\lambda_{n+1}} \right) ^{2}}
= 2\,\nu \gamma \, \frac{1 - \mu}{(1 + \mu)^2}
\, \lambda \geq 2\,\nu\gamma\frac{1-\mu}{(1+\mu)^{2}}\min\left\{ \lambda_1, \frac{\mu}{L} \right\}.$$
Therefore there exists  $N\in\mathbb{N}$ such that for all $n>N$, we have
\begin{align}\label{eq:34}
    \| u_n - x^\ast \|^2 &\leq \| w_n - x^\ast\|^2 -\frac{1}{2}\min\left\{
\frac{(2\gamma - \gamma^2)(1 - \mu)^2}{(1 + \mu)^2}, \;
2\,\nu \gamma\frac{1 - \mu}{(1 + \mu)^2}\min\left\{ \lambda_1, \frac{\mu}{L} \right\}
\right\}\left(\|w_{n}-y_{n}\|^2+\|y_{n}-x^\ast\|^2\right) \notag\\
   &\leq \|w_{n}-x^\ast\|^2 - \frac{1}{4}\min\left\{
\frac{(2\gamma - \gamma^2)(1 - \mu)^2}{(1 + \mu)^2}, \;
2\,\nu \gamma\frac{1 - \mu}{(1 + \mu)^2}\min\left\{ \lambda_1, \frac{\mu}{L} \right\}
\right\}\| w_n - x^\ast\|^2\notag\\
&= \left[1 - \frac{1}{4}\min\left\{
\frac{(2\gamma - \gamma^2)(1 - \mu)^2}{(1 + \mu)^2}, \;
2\,\nu \gamma\frac{1 - \mu}{(1 + \mu)^2}\min\left\{ \lambda_1, \frac{\mu}{L} \right\}
\right\}\right] \| w_n - x^\ast\|^2\notag\\
&=\tau \, \| w_n - x^\ast \|^2.
\end{align}
}
Using \eqref{eq:10}, \eqref{eq:34} in \eqref{lineardefxn}, we get
\begin{align}
\|x_{n+1} - x^\ast\|^2 
&= \|(1 - \alpha)(x_n - x^\ast) + \alpha(u_n - x^\ast)\|^2 \nonumber \\
&= (1 - \alpha)\|x_n - x^\ast\|^2 + \alpha\|u_n - x^\ast\|^2 - \alpha(1 - \alpha)\|x_n - u_n\|^2 \nonumber \\
&\leq (1 - \alpha)\|x_n - x^\ast\|^2 + \alpha\, \tau \|w_n - x^\ast\|^2 - \alpha(1 - \alpha)\|x_n - u_n\|^2 \nonumber \\
&\leq (1 - \alpha)\|x_n - x^\ast\|^2 - \frac{1 - \alpha}{\alpha} \|x_{n+1} - x_n\|^2 \nonumber \\
&\quad + \alpha\, \tau \left[ (1 + \theta) \|x_n - x^\ast\|^2 - \theta \|x_{n-1} - x^\ast\|^2 + \theta(1 + \theta) \|x_n - x_{n-1}\|^2 \right] \nonumber \\
&\leq \left[1 - \alpha(1 - \tau(1 + \theta)) \right]\, \|x_n - x^\ast\|^2 - \alpha\, \tau\, \theta \|x_{n-1} - x^\ast\|^2 \nonumber \\
&\quad + \alpha\, \tau\, \theta(1 + \theta) \|x_n - x_{n-1}\|^2 
- \frac{1 - \alpha}{\alpha}\, \|x_{n+1} - x_n\|^2 ,
\end{align}
which implies that
\begin{align}\label{eq:36}
\|x_{n+1} - x^\ast\|^2 + \frac{1 - \alpha}{\alpha} \|x_{n+1} - x_n\|^2 
&\leq \left[1 - \alpha(1 - \tau(1 + \theta)) \right] \|x_n - x^\ast\|^2 \\
&\quad - \alpha \tau \theta \|x_{n-1} - x^\ast\|^2 
+ \alpha \tau \theta(1 + \theta) \|x_n - x_{n-1}\|^2.
\end{align}
Since $0 < \alpha < \tfrac{1}{3}$, we have from \eqref{eq:36} that
\begin{align}
\|x_{n+1} - x^\ast\|^2 + \|x_{n+1} - x_n\|^2 
&\leq \left[1 - \alpha(1 - \tau(1 + \theta))\right] \|x_n - x^\ast\|^2 
+ \alpha \tau \theta(1 + \theta)\|x_n - x_{n-1}\|^2 \nonumber \\
&\leq \left[1 - \alpha(1 - \tau(1 + \theta))\right]
\left[\|x_n - x^\ast\|^2 
+ \frac{\alpha \tau \theta(1 + \theta)}{1 - \alpha(1 - \tau(1 + \theta))}
    \|x_n - x_{n-1}\|^2 \right]. \nonumber
\end{align}
Using the fact that
\[
\frac{\alpha \tau \theta(1 + \theta)}{1 - \alpha(1 - \tau(1 + \theta))} < 1,
\]
we obtain
\begin{equation}\label{eq:37}
\|x_{n+1} - x^\ast\|^2 + \|x_{n+1} - x_n\|^2
\leq \left[1 - \alpha(1 - \tau(1 + \theta))\right]
\left[\|x_n - x^\ast\|^2 + \|x_n - x_{n-1}\|^2\right].
\end{equation}

{From (A3$'$), $\tau(1+\theta) < 1$, so $0 < 1 - \tau(1+\theta) < 1$. With $\alpha \in (0, 1/3)$ from (A4$'$), we obtain $0 < \alpha(1 - \tau(1+\theta)) < 1$, hence $0 < 1 - \alpha(1 - \tau(1+\theta)) < 1$.}

Now, define
$
b_n := \|x_n - x^\ast\|^2 + \|x_n - x_{n-1}\|^2, \quad \forall n \geq 1.
$
Then from~\eqref{eq:37}, we obtain
$
b_{n+1} \leq \left[1 - \alpha(1 - \tau(1 + \theta))\right] b_n.
$
By induction, it follows that
$$
b_{n+1} \leq \left[1 - \alpha(1 - \tau(1 + \theta))\right]^n b_1.
$$
Therefore, by the definition of $b_n$,
$$
\|x_{n+1} - x^\ast\|^2 \leq \left[1 - \alpha(1 - \tau(1 + \theta))\right]^n b_1.
$$
This concludes the proof.
\end{proof}

{
\subsection{Strong Convergence Without Strong Monotonicity}\label{S:halpern}

In this subsection, we introduce a variant of Algorithm~\ref{algo:IPCMAS1}  that achieves strong convergence without requiring strong monotonicity of~$A$.

\begin{assumption}\label{assumption:halpern}
In addition to Assumption~\ref{assumption}~\textup{(A1)--(A3)}, suppose that
\begin{itemize}
  \item[\textnormal{(A4$^{\prime\prime}$)}] The sequences $\{\alpha_{n}\}, \{\sigma_{n}\} \subset (0,1)$ satisfy
  \(
  \alpha_{n} \to 0,
  ~
  \sum\limits_{n=1}^{\infty} \alpha_n = +\infty,
  \)
  and $\{\sigma_{n}\}\subseteq(a, b)\subset(0, 1-\alpha_{n})$ for some $a, b>0$.

  \item[\textnormal{(A5$^{\prime\prime}$)}] The sequences  $\{\varepsilon_n\}$ and $\{\varepsilon'_n\}$   are in $(0,\infty)$ such that $\varepsilon_n\leq \varepsilon'_n$ and
$\sum\limits_{n=1}^{\infty}\frac{\varepsilon'_{n}}{\alpha_{n}}<+\infty.$

  \item[\textnormal{(A6$^{\prime\prime}$)}] The constants $\bar\beta, \bar\theta$ satisfy $0 \leq \bar\beta \leq \bar\theta \leq 1$.
\end{itemize}
\end{assumption}

For each $n\ge 1$, define the adaptive inertial parameters by
\begin{equation}\label{eq:newbeta'n}
       \beta_n'=
\begin{cases}
\min\left\{\bar\beta,\dfrac{\varepsilon_n}{\|x_n-x_{n-1}\|},\, \dfrac{\varepsilon_n}{\|x_n-x_{n-1}\|^2}\right\},
& \text{if }x_n\neq x_{n-1},\\[2ex]
\bar\beta, & \text{otherwise},
\end{cases} 
\end{equation}
and
\begin{equation}\label{eq:newthetan}
    \theta_n=
\begin{cases}
\min\left\{\bar\theta,\dfrac{\varepsilon'_n}{\|x_n-x_{n-1}\|}\,, \dfrac{\varepsilon'_n}{\|x_n-x_{n-1}\|^2 }\right\},
& \text{if }x_n\neq x_{n-1},\\[2ex]
\bar\theta, & \text{otherwise}.
\end{cases}
\end{equation}
Since \(\varepsilon_n \leq \varepsilon'_n\) and \(\bar\beta \leq \bar\theta\), it follows from \eqref{eq:newbeta'n} and \eqref{eq:newthetan} that
\(\beta'_n \leq \theta_n\). Also, we have 
\[
\frac{\beta'_{n}}{\alpha_{n}}\|x_{n}-x_{n-1}\|\leq \frac{\varepsilon_{n}}{\alpha_{n}}\leq\frac{\varepsilon'_{n}}{\alpha_{n}}
\quad \mbox{and} \quad 
\frac{\beta'_{n}}{\alpha_{n}}\|x_{n}-x_{n-1}\|^{2}\leq \frac{\varepsilon_{n}}{\alpha_{n}}\ \leq\frac{\varepsilon'_{n}}{\alpha_{n}},
\]
which, together with \((A5^{\prime\prime})\), implies that
\begin{equation}\label{eq:sumauxeq}
   \sum\limits_{n=1}^{\infty}\frac{\beta'_{n}}{\alpha_{n}}\|x_{n}-x_{n-1}\|<+\infty
   \quad \mbox{and} \quad
   \sum\limits_{n=1}^{\infty}\frac{\beta'_{n}}{\alpha_{n}}\|x_{n}-x_{n-1}\|^2<+\infty.
\end{equation}
Consequently,
\begin{equation}\label{eq:fraclimauxeq}
   \lim_{n\to\infty}\frac{\beta'_{n}}{\alpha_{n}}\|x_{n}-x_{n-1}\|=0
   \quad \mbox{and} \quad
   \lim_{n\to\infty}\frac{\beta'_{n}}{\alpha_{n}}\|x_{n}-x_{n-1}\|^2=0.
\end{equation}
Moreover, since \(\alpha_{n}\in(0,1)\), we also have
\begin{equation}\label{eq:limauxeq}
    \lim_{n\to\infty}\beta'_{n}\|x_{n}-x_{n-1}\|=0
    \quad \mbox{and} \quad
    \lim_{n\to\infty}\beta'_{n}\|x_{n}-x_{n-1}\|^2=0.
\end{equation}
By following the above argument and \eqref{eq:newthetan}, we also have
\begin{equation}\label{eq:88}
   \sum\limits_{n=1}^{\infty}\frac{\theta_{n}}{\alpha_{n}}\|x_{n}-x_{n-1}\|<+\infty
   \quad \mbox{and} \quad
   \sum\limits_{n=1}^{\infty}\frac{\theta_{n}}{\alpha_{n}}\|x_{n}-x_{n-1}\|^{2}<+\infty,
\end{equation}
 and hence
\begin{equation}\label{eq:99}
   \lim_{n\to\infty}\frac{\theta_{n}}{\alpha_{n}}\|x_{n}-x_{n-1}\|=0
   \quad \mbox{and} \quad
   \lim_{n\to\infty}\frac{\theta_{n}}{\alpha_{n}}\|x_{n}-x_{n-1}\|^2=0.
\end{equation}
Since \(\alpha_{n}\in(0,1)\), we also have
\begin{equation}\label{eq:100}
    \lim_{n\to\infty}{\theta_{n}}\|x_{n}-x_{n-1}\|=0
    \quad \mbox{and} \quad
    \lim_{n\to\infty}\theta_{n}\|x_{n}-x_{n-1}\|^2=0.
\end{equation}
\begin{algorithm}[H]
\begin{footnotesize}
\begin{description}
  \item[Step 0. ]
    Choose $\mu \in (0,1)$, $\gamma \in (0,2)$, $\lambda_1 > 0$, and parameters satisfying Assumption~\ref{assumption:halpern}.
    Choose starting points $x_0, x_1 \in \mathscr{H}$ and set $n := 1$.

  \item[Step 1.] Compute
  \begin{equation}\label{newhalpern:znwnyn}
\begin{cases}
z_n = x_n + \beta_n'(x_n - x_{n-1}),\\[0.3em]
w_n = x_n + \theta_n(x_n - x_{n-1}),\\[0.3em]
y_n = J^A_{\lambda_n}(w_n - \lambda_n B(w_n)).
\end{cases}
\end{equation}
    If $y_n = w_n$, then stop; $y_n \in \Omega$. Otherwise go to Step~2.

  \item[Step 2.] Compute $u_n = w_n - \gamma\,\eta_n\,d_n$, where $d_n$ and $\eta_n$ are as in \eqref{defdn}--\eqref{defrhon}.

  \item[Step 3.] Compute
    \begin{equation}\label{mann:xn}
      x_{n+1} = (1-\alpha_{n}-\sigma_{n})\, z_n + \sigma_n\, u_n.
    \end{equation}

  \item[Step 4.] Update $\lambda_{n+1}$ as in \eqref{deflambda}. Set $n := n+1$ and go to Step~1.
\end{description}
\caption{Double Inertial PCM with Implicit Contraction (\algoDIM\!\!)}
\label{algo:DIPCM}
\end{footnotesize}
\end{algorithm}

\begin{theorem}\label{thm:halpern-strong}
Suppose that Assumption~{\rm\ref{assumption:halpern}} holds, and let $\{x_n\}$ be the sequence generated by Algorithm~{\rm\ref{algo:DIPCM}}. Then $\{x_n\}$ converges strongly to
$
x^* = P_{\Omega}(0).
$
\end{theorem}

\begin{proof}

\textbf{Claim 1.} The sequence $\{x_{n}\}$ is bounded.

It follows from the definition of $z_{n}$  in \eqref{newhalpern:znwnyn}
\begin{align}\label{eq:zn-est-halpern}
\|z_n-x^*\|^2
&=\|x_n+\beta_n'(x_n-x_{n-1})-x^*\|^2 \notag\\
&=\|x_{n}-x^*\|^{2}+\beta_n'^{2}\|x_n-x_{n-1}\|^2+2\beta_{n}'\left\langle x_n-x^*,\, x_n-x_{n-1}\right\rangle\notag\\
&=\|x_{n}-x^*\|^{2}+\beta_n'^{2}\|x_n-x_{n-1}\|^2+\beta_{n}'\left(\|x_{n}-x^*\|^2+\|x_{n}-x_{n-1}\|^2-\|x_{n-1}-x^*\|^2\right)\notag\\
&=\|x_n-x^*\|^2
+\beta_n'\bigl(\|x_n-x^*\|^2-\|x_{n-1}-x^*\|^2\bigr)
+\beta_n'(1+\beta_n')\|x_n-x_{n-1}\|^2. 
\end{align}
Similarly,
\begin{align}
\|w_n-x^*\|^2
&=\|x_n+\theta_n(x_n-x_{n-1})-x^*)\|^2 \notag\\
&=\|x_n-x^*\|^2
+\theta_n\bigl(\|x_n-x^*\|^2-\|x_{n-1}-x^*\|^2\bigr)
+\theta_n(1+\theta_n)\|x_n-x_{n-1}\|^2. \label{eq:wn-est-halpern}
\end{align}

On combining \eqref{eq:27} and \eqref{eq:wn-est-halpern}, we obtain
\begin{align}
\|u_n-x^*\|^2&
\leq\| w_n - x^\ast \|^2
-\frac{(2-\gamma)}{\gamma}\|u_n-w_n\|^2\notag\\
&\le \|x_n-x^*\|^2
+\theta_n\bigl(\|x_n-x^*\|^2-\|x_{n-1}-x^*\|^2\bigr)
+\theta_n(1+\theta_n)\|x_n-x_{n-1}\|^2\notag\\
&\quad
-\frac{(2-\gamma)}{\gamma}\|u_n-w_n\|^2 . \label{eq:un-est-halpern}
\end{align}

It follows from the definition of \(x_{n+1}\) in \eqref{mann:xn}, together with
\eqref{eq:zn-est-halpern} and \eqref{eq:un-est-halpern}, that
\begin{align}
\|x_{n+1}-x^*\|^2
&=\|(1-\alpha_n-\sigma_n)z_n+\sigma_n u_n-x^*\|^2 \notag\\
&=\|(1-\alpha_n-\sigma_n)(z_n-x^*)+\sigma_n(u_n-x^*)-\alpha_n x^*\|^2 \notag\\
&\le (1-\alpha_n-\sigma_n)\|z_n-x^*\|^2
   +\sigma_n\|u_n-x^*\|^2
   +\alpha_n\|x^*\|^2 \notag\\
&\le (1-\alpha_n-\sigma_n)\Bigl[\|x_n-x^*\|^2
+\beta_n'\bigl(\|x_n-x^*\|^2-\|x_{n-1}-x^*\|^2\bigr)
+\beta_n'(1+\beta_n')\|x_n-x_{n-1}\|^2\Bigr] \notag\\
&\quad +\sigma_n\Bigl[\|x_n-x^*\|^2
+\theta_n\bigl(\|x_n-x^*\|^2-\|x_{n-1}-x^*\|^2\bigr)
+\theta_n(1+\theta_n)\|x_n-x_{n-1}\|^2\notag\\\
&\quad-\frac{2-\gamma}{\gamma}\|u_n-w_n\|^2\Bigr]
+\alpha_n\|x^*\|^2 \notag\\
&= (1-\alpha_n)\|x_n-x^*\|^2  +\Bigl[(1-\alpha_n-\sigma_n)\beta_n'+\sigma_n\theta_n\Bigr]
\bigl(\|x_n-x^*\|^2-\|x_{n-1}-x^*\|^2\bigr) \notag\\
&\quad +\Bigl[(1-\alpha_n-\sigma_n)\beta_n'(1+\beta_n')
+\sigma_n\theta_n(1+\theta_n)\Bigr]\|x_n-x_{n-1}\|^2 \notag\\
& -\frac{2-\gamma}{\gamma}\sigma_n\|u_n-w_n\|^2
+\alpha_n\|x^*\|^2 .
\label{eq:au1-new}
\end{align}

Since \(0\le \beta_n'\le \theta_n\), we have
\(
(1-\alpha_n-\sigma_n)\beta_n'(1+\beta_n')
\le (1-\alpha_n-\sigma_n)\theta_n(1+\theta_n),
\)
and hence
\begin{align}\label{eq:firstupperbdd}
(1-\alpha_n-\sigma_n)\beta_n'(1+\beta_n')
+\sigma_n\theta_n(1+\theta_n) 
&\le
(1-\alpha_n-\sigma_n)\theta_n(1+\theta_n)
+\sigma_n\theta_n(1+\theta_n)\notag \\
&=(1-\alpha_n)\theta_n(1+\theta_n)\notag\\
&\leq2\,\theta_{n}\,(1-\alpha_{n}).
\end{align}
Moreover, since
$
(1-\alpha_n-\sigma_n)\beta_n'+\sigma_n\theta_n\ge 0
$ and $\beta'_{n}\leq\theta_{n}$, we have
\begin{align}\label{eq:secondupprbdd}
    &\bigl[(1-\alpha_n-\sigma_n)\beta_n'+\sigma_n\theta_n\bigr]
\bigl(\|x_n-x^*\|^2-\|x_{n-1}-x^*\|^2\bigr)\notag\\
&\quad\leq\bigl[(1-\alpha_n-\sigma_n)\beta_n'+\sigma_n\theta_n\bigr]\,\bigl[\|x_n-x^*\|^2-\|x_{n-1}-x^*\|^2\bigr]_+\notag\\
&\quad\le
(1-\alpha_n)\theta_n
\bigl[\|x_n-x^*\|^2-\|x_{n-1}-x^*\|^2\bigr]_+.
\end{align}
Substituting these estimates into \eqref{eq:au1-new}, we obtain
\begin{align}
\|x_{n+1}-x^*\|^2
&\le (1-\alpha_n)\|x_n-x^*\|^2
 +(1-\alpha_n)\theta_n
\bigl[\|x_n-x^*\|^2-\|x_{n-1}-x^*\|^2\bigr]_+ \notag\\
&\quad +2\,\theta_n\,(1-\alpha_n)\|x_n-x_{n-1}\|^2
-\frac{2-\gamma}{\gamma}\sigma_n\|u_n-w_n\|^2
+\alpha_n\|x^*\|^2\label{eq:bdd1}\\
&\le (1-\alpha_n)\Bigl[
\|x_n-x^*\|^2
+\theta_n\bigl[\|x_n-x^*\|^2-\|x_{n-1}-x^*\|^2\bigr]_+
+2\,\theta_n\|x_n-x_{n-1}\|^2
\Bigr] \notag\\
&\quad+\alpha_n\|x^*\|^2 \notag\\
&\le
\max\Bigl\{
\|x^*\|^2,\,
\|x_n-x^*\|^2
+\theta_n\bigl[\|x_n-x^*\|^2-\|x_{n-1}-x^*\|^2\bigr]_+
+2\,\theta_n\|x_n-x_{n-1}\|^2
\Bigr\}.
\label{eq:bdd2}
\end{align}

We now consider two cases.

\medskip
\noindent
\textbf{Case I.}
Suppose that
\(
\|x_{n+1}-x^*\|^2\le \|x^*\|^2
\). Then, the sequence
\(\{\|x_n-x^*\|\}\) is bounded.  Hence \(\{x_n\}\) is bounded.

\medskip
\noindent
\textbf{Case II.}
Suppose that
\begin{align}
\|x_{n+1}-x^*\|^2
\le
\|x_n-x^*\|^2
+\theta_n\bigl[\|x_n-x^*\|^2-\|x_{n-1}-x^*\|^2\bigr]_+
+2\,\theta_{n}\|x_n-x_{n-1}\|^2.
\label{eq:case2}
\end{align}
 Since $\alpha_{n}\in(0,1)$, it follows from \eqref{eq:88} that $\sum\limits_{n=1}^{\infty}\theta_{n}\|x_{n}-x_{n-1}\|^{2}< +\infty$. Thus, from Lemma \ref{lemma:auxiliaryconvergence}, it follows that $\lim\limits_{n\to\infty}\|x_{n+1}-x^*\|$ exists and hence $\|x_{n+1}-x^*\|$ is bounded, consequently $\{x_{n}\}$ is bounded.

Therefore, in either case, the sequence \(\{x_n\}\) is bounded.

\textbf{Claim 2.} For all \(n\ge 1\), we prove that
\begin{align}\label{eq:claim2-halpern}
a\left(\frac{2-\gamma}{\gamma}\right)\|u_n-w_n\|^2
&\le \|x_n-x^*\|^2-\|x_{n+1}-x^*\|^2 +\theta_n\bigl[\|x_n-x^*\|^2-\|x_{n-1}-x^*\|^2\bigr]_+ \notag\\
&\quad +2\theta_n\|x_n-x_{n-1}\|^2+\alpha_n M_2,
\end{align}
where
$
M_2:=\sup_{n\ge 1}\left|\|x^*\|^2-\|x_n-x^*\|^2\right|<+\infty.
$

Indeed, since $\sigma_n\ge a$ and $\alpha_{n}\in(0,1)$, from \eqref{eq:bdd1} we obtain
\begin{align*}
a\left(\frac{2-\gamma}{\gamma}\right)\|u_n-w_n\|^2
&\le \sigma_n\left(\frac{2-\gamma}{\gamma}\right)\|u_n-w_n\|^2 \\
&\le \|x_n-x^*\|^2-\|x_{n+1}-x^*\|^2
+\theta_n\bigl[\|x_n-x^*\|^2-\|x_{n-1}-x^*\|^2\bigr]_+ \\
&\quad +2\theta_n\|x_n-x_{n-1}\|^2
+\alpha_n\bigl(\|x^*\|^2-\|x_n-x^*\|^2\bigr).
\end{align*}
which proves \eqref{eq:claim2-halpern}.

\textbf{Claim 3.} We prove that
\begin{align*}
\|x_{n+1}-x^*\|^2
&\le (1-\alpha_n)\|x_n-x^*\|^2
+\theta_n
\bigl[\|x_n-x^*\|^2-\|x_{n-1}-x^*\|^2\bigr]_{+}\notag\\
&\quad+ 2\,\theta_n
\|x_n-x_{n-1}\|^2
+ 2\,\alpha_n\, \langle x^*-x_{n+1},\, x^*\rangle \\
&= (1-\alpha_n)\|x_n-x^* \|^2+\alpha_n\,\Delta_{n},\quad \forall n\geq 1
\end{align*}
where $\Delta_n
:= \frac{\theta_n}{\alpha_n}
\Big[\|x_n-x^*\|^2-\|x_{n-1}-x^*\|^2\Big]_+
+ \frac{2\theta_{n}}{\alpha_n}\|x_n-x_{n-1}\|^2
+ 2\langle x^*-x_{n+1},\, x^*\rangle.$

From Lemma~\ref{lemma:inner_product_identity}(iii) and \eqref{eq:27}, \eqref{eq:zn-est-halpern}, \eqref{eq:wn-est-halpern}, we obtain
\begin{align}
\|x_{n+1}-x^*\|^2
&= \|(1-\alpha_n-\sigma_n)z_n + \sigma_n u_n - x^*\|^2 \notag\\
&= \left\|(1-\alpha_n-\sigma_n)z_n + \sigma_n u_n + \alpha_n x^* - x^*
+\alpha_n(-x^*)\right\|^2 \notag\\
&\le \|(1-\alpha_n-\sigma_n)z_n + \sigma_n u_n + \alpha_n x^* - x^*\|^2
- 2\alpha_n\langle x_{n+1}-x^*, x^*\rangle \notag\\
&\le \left\|(1-\alpha_n-\sigma_n)(z_n-x^*) + \sigma_n(u_n-x^*)+\alpha_{n}(x^*-x^*)\right\|^2
- 2\alpha_n\langle x_{n+1}-x^*, x^*\rangle \notag\\
&\le (1-\alpha_n-\sigma_n)\|z_n-x^*\|^2
+ \sigma_n\|u_n-x^*\|^2+\alpha_{n}\|x^*-x^*\|^{2}
- 2\alpha_n\langle x_{n+1}-x^*, x^*\rangle \notag\\
&\le (1-\alpha_n-\sigma_n)\Big[
\|x_n-x^*\|^2
+ \beta'_n(\|x_n-x^*\|^2 - \|x_{n-1}-x^*\|^2)_{+} +\beta'_n(1+\beta'_{n}) \|x_n-x_{n-1}\|^2
 \Big] \notag \\
 &\quad+\sigma_n \Big[
\|x_n-x^*\|^2
+ \theta_n(\|x_n-x^*\|^2 - \|x_{n-1}-x^*\|^2)_{+}+ \theta_n (1+\theta_{n})\|x_n-x_{n-1}\|^2 \Big]\notag\\
&\quad- 2\alpha_n \langle x_{n+1} - x^*, x^* \rangle-\sigma_{n}\left(\frac{2-\gamma}{\gamma}\right)\|u_n-w_n\|^2 \notag\\
&\le (1-\alpha_n)\|x_n-x^*\|^2+ \bigl[(1-\alpha_n-\sigma_n)\beta_n'+\sigma_n\theta_n\bigr]
[\|x_n-x^*\|^2 - \|x_{n-1}-x^*\|^2]_{+} \notag\\
&\quad + \Big((1-\alpha_n-\sigma_n)\beta_n'(1+\beta_n')
+\sigma_n\theta_n(1+\theta_n)
\Big)\|x_n-x_{n-1}\|^2 \notag\\
&\quad - 2\alpha_n \langle x_{n+1} - x^*, x^* \rangle -\sigma_{n}\left(\frac{2-\gamma}{\gamma}\right)\|u_n-w_n\|^2. \notag
\end{align}
Thus, taking into account that \(\gamma \in (0,2)\) and using \eqref{eq:firstupperbdd} and \eqref{eq:secondupprbdd}, we get
\begin{align}
    \|x_{n+1}-x^*\|^2
    &\le (1-\alpha_n)\|x_n-x^*\|^2 + (1-\alpha_{n})\,\theta_n(\|x_n-x^*\|^2 - \|x_{n-1}-x^*\|^2)_{+} \notag\\
&\quad + 2\,\theta_{n}(1-\alpha_{n})\,\|x_n-x_{n-1}\|^2  - 2\alpha_n \langle x_{n+1}-x^*, x^* \rangle\notag\\
&\le (1-\alpha_n)\|x_n-x^*\|^2
+ \theta_n(\|x_n-x^*\|^2 - \|x_{n-1}-x^*\|^2)_{+} + 2\theta_{n}\|x_n-x_{n-1}\|^2\notag\\
& \quad+2\alpha_n \langle x^*-x_{n+1}, x^* \rangle.
\end{align}

\textbf{Claim 4.} Finally, it remains to prove that $\{\|x_{n}-x^*\|\}$ converges to zero. To prove this, by Lemma \ref{lmm:convergence} it suffices to verify that
$$\limsup_{k\to\infty} \left(
\frac{\theta_{n_k}}{\alpha_{n_k}}
\Big[\|x_{n_k}-x^*\|^2-\|x_{n_k-1}-x^*\|^2\Big]_+
+ \frac{2\theta_{n_k}}{\alpha_{n_k}}\|x_{n_k}-x_{n_k-1}\|^2
+ 2\langle x^*-x_{n_k+1},\, x^*\rangle
\right)\leq0$$
for every subsequence $\{\|x_{n_{k}}-x^*\|^{2}\}$ of $\{\|x_{n}-x^*\|\}$ satisfying
$$\liminf_{k\to\infty}\big(\|x_{n_{k}+1}-x^*\|^{2}- \|x_{n_{k}}-x^*\|^{2}\big) \ge 0.$$  For this purpose suppose that $\{\|x_{n_{k}}-x^*\|^{2}\}$ is a subsequence of $\{\|x_{n}-x^*\|^{2}\}$ such that
$$\liminf_{k\to\infty}\big(\|x_{n_{k}+1}-x^*\|^{2}- \|x_{n_{k}}-x^*\|^{2}\big) \ge 0.$$   First note that
\begin{align*}
\frac{\theta_{n_k}}{\alpha_{n_k}}
\Big(\|x_{n_k}-x^*\|^2-\|x_{n_k-1}-x^*\|^2\Big)_{+}
&\leq\frac{\theta_{n_k}}{\alpha_{n_k}}
\Big(\Bigl|\|x_{n_k}-x^*\|^2-\|x_{n_k-1}-x^*\|^2\Bigr|\Big)\notag\\
&\leq
\frac{\theta_{n_k}}{\alpha_{n_k}}\|x_{n_k}-x_{n_k-1}\|^2
+ \frac{2\theta_{n_k}}{\alpha_{n_k}}\Bigl|\langle x_{n_k}-x^*,\;x_{n_k}-x_{n_k-1}\rangle\Bigr| \\ 
&\le \frac{\theta_{n_k}}{\alpha_{n_k}}\|x_{n_k}-x_{n_k-1}\|^2+ \frac{2\theta_{n_k}}{\alpha_{n_k}}\|x_{n_k}-x_{n_k-1}\|\,M_1,
\end{align*}
where $M_{1}:=\sup_{n\geq 0}\{\|x_{n}-x^*\|\}$, since $\{x_{n}\}$ is bounded. Hence, by~\eqref{eq:99}
it follows that
$
\lim\limits_{k\to\infty}\frac{\theta_{n_k}}{\alpha_{n_k}
}\Bigl(\|x_{n_k}-x^*\|^2-\|x_{n_k-1}-x^*\|^2\Bigr)_{+}=0.
$
Moreover, since \(\alpha_{n_k}\in (0,1)\), we have
$
\theta_{n_k}
\Bigl(\|x_{n_k}-x^*\|^2-\|x_{n_k-1}-x^*\|^2\Bigr)_{+}
\le
\frac{\theta_{n_k}}{\alpha_{n_k}}
\Bigl(\|x_{n_k}-x^*\|^2-\|x_{n_k-1}-x^*\|^2\Bigr)_{+}.
$
Thus,
\begin{equation}\label{e}
\lim_{k \to \infty} \theta_{n_k}
\Bigl(\|x_{n_k}-x^*\|^2-\|x_{n_k-1}-x^*\|^2\Bigr)_{+} = 0.
\end{equation}
Then, from Claim 2, \eqref{eq:limauxeq} and \eqref{e}, we obtain
\begin{align*}
a\,\left(\frac{2-\gamma}{\gamma}\right)\limsup_{k\to\infty}\|u_{n_{k}}-w_{n_{k}}\|^{2}&
\le
\limsup_{k\to\infty}
\Bigl(
\|x_{n_k}-x^*\|^2-\|x_{n_k+1}-x^*\|^2
\Bigr) \notag\\
&
+
\limsup_{k\to\infty}
\Bigl(
\theta_{n_k}
\bigl[\|x_{n_k}-x^*\|^2-\|x_{n_k-1}-x^*\|^2\bigr]_+
\Bigr) \\
&\quad+
\limsup_{k\to\infty}
\Bigl(
2\theta_{n_k}\|x_{n_k}-x_{n_k-1}\|^2
\Bigr)
+
\limsup_{k\to\infty} 2\alpha_{n_k}M_1 \notag\\
&=-\liminf_{k\to\infty}
\Bigl(
\|x_{n_k}-x^*\|^2-\|x_{n_k+1}-x^*\|^2
\Bigr) \notag\\
&\leq0.
\end{align*}
This implies that 
\begin{equation}\label{eq:1}
    \lim\limits_{n\to\infty}\|u_{n}-w_{n}\|=0.
\end{equation}
Thus, from \eqref{eq:27} and taking into account $\lim\limits_{n\to\infty}\lambda_{n}=\lambda$, we get
\begin{equation}\label{eq;limunwn}
    \lim\limits_{n\to\infty}\|w_{n}-y_{n}\|=0.
\end{equation}

On the other hand, from \eqref{eq:limauxeq},  \eqref{eq:100} and  \eqref{newhalpern:znwnyn},  we get

\begin{equation}
    \lim_{k\to\infty} \|z_{n_{k}} - x_{n_k}\|=
\lim_{k\to\infty} \beta_{n_{k}}'\|x_{n_{k}} - x_{n_k-1}\|
= 0.
\end{equation}
and 
\begin{equation}\label{eq:2}
 \lim_{k\to\infty} \|w_{n_{k}} - x_{n_k}\|
=
\lim_{k\to\infty} \theta_{n_{k}}\|x_{n_{k}} - x_{n_k-1}\|
= 0.   
\end{equation}
Thus, we have
\begin{equation}\label{eq:3}
   \|z_{n_{k}}-w_{n_{k}}\|\leq \|z_{n_{k}}-x_{n_{k}}\|+\|w_{n_{k}}-x_{n_{k}}\|\to 0. 
\end{equation}
It follows from \eqref{eq:1}, \eqref{eq:2} and \eqref{eq:3} that
\begin{align*}
\|x_{n_k+1}-x_{n_k}\|
&\le
\|x_{n_k+1}-z_{n_k}\|
+\|z_{n_k}-x_{n_k}\|\\
&\le
\sigma_{n_k}\|u_{n_k}-z_{n_k}\|+ \alpha_{n_k}\|z_{n_k}\|
+\|z_{n_k}-x_{n_k}\|\\
&\le
\sigma_{n_k}\bigl(\|u_{n_k}-w_{n_k}\|+\alpha_{n_k}\|z_{n_k}\|
+\|w_{n_k}-z_{n_k}\|\bigr)
+\|z_{n_k}-x_{n_k}\|
\to 0.
\end{align*}
Thus
\begin{equation}\label{eq:limxnk}
    \lim_{k\to\infty}\|x_{n_k+1}-x_{n_k}\|=0.
\end{equation}
Moreover, from \eqref{eq;limunwn} and \eqref{eq:2}, we obtain
\begin{align*}
\|x_{n_k}-y_{n_k}\|
&\le \|x_{n_k}-w_{n_k}\|+\|w_{n_k}-y_{n_k}\|
\to 0.
\end{align*}

Since $\{x_{n_k}\}$ is bounded, the scalar sequence $\{\langle x^*, x^* - x_{n_k}\rangle\}$ is also bounded. Choose a subsequence $\{x_{n_{k_j}}\}$ of $\{x_{n_k}\}$ along which
\[
\lim_{j\to\infty} \langle x^*, x^* - x_{n_{k_j}} \rangle
= \limsup_{k\to\infty} \langle x^*, x^* - x_{n_k} \rangle .
\]
Since $\{x_{n_{k_j}}\}$ is bounded, by passing to a further subsequence (without relabeling) we may assume $x_{n_{k_j}} \rightharpoonup \bar{x}$ for some $\bar{x} \in \mathscr{H}$. Hence
\begin{equation}\label{eq:mm}
    \limsup_{k\to\infty} \langle x^*, x^* - x_{n_k} \rangle
=
\lim_{j\to\infty} \langle x^*, x^* - x_{n_{k_j}} \rangle
=
\langle x^*, x^* - \bar{x}\rangle .
\end{equation}
Since $\lim\limits_{j\to\infty}\|x_{n_{k_j}}-w_{n_{k_j}}\|=0$, we have $w_{n_{k_j}}\rightharpoonup\bar{x}$. Moreover, $\lim\limits_{j\to\infty}\|x_{n_{k_j}}-y_{n_{k_j}}\| = 0$ together with the argument in the last paragraph of Theorem \ref{th:weakconv} implies that $\bar{x} \in \Omega$. Hence, by \eqref{eq:mm} and the definition of $x^* = P_{\Omega}(0)$, we have
\begin{equation}\label{eq:bb}
\limsup_{k\to\infty} \langle x^*, x^* - x_{n_k} \rangle
=
\langle x^*, x^* - \bar{x}\rangle \le 0.
\end{equation}
On combining \eqref{eq:limxnk} and \eqref{eq:bb}
\begin{equation}\label{eq:suplim}
\limsup_{k\to\infty} \langle x^*, x^* - x_{n_k+1} \rangle
\le
\limsup_{k\to\infty} \langle x^*, x^* - x_{n_k} \rangle
=
\langle x^*, x^* - \bar{x} \rangle
\le 0.
\end{equation}
Thus, from \eqref{eq:99}, \eqref{e} and \eqref{eq:suplim}, we get $\limsup_{k\to\infty} \Delta_{n_k} \leq0$. Hence, from Claim 3 and Lemma~\ref{lmm:convergence},
we  have $\lim\limits_{n\to\infty}\|x_{n}-x^*\|=0$. This completes the proof.
\end{proof}

}


\section{Numerical Experiments}\label{S:5}


{In this section, we present numerical experiments to evaluate the performance of the proposed
algorithms.} {The numerical study is organized as follows. Section~\ref{subsec:example1} (Example~1, classical VIP) compares all four algorithms --\algoDIM, \algoT, \algoD, and \algoS-- at varying problem dimensions and tolerances. Sections~\ref{subsec:example_2} (split feasibility, $\varepsilon = 10^{-3}$) and~\ref{subsec:elastic-net} (elastic net regularization, $\varepsilon = 10^{-6}$ with a $50{,}000$-iteration budget) carry out the same four-way comparison on these problem classes. Section~\ref{sub:r_linearity} returns to \algoT to verify its $R$-linear convergence on the (strongly monotone) elastic-net problem, the regime for which it is theoretically designed. Algorithm~\ref{algo:IPCMAS1} (\algoO\!\!) is excluded throughout: we consider only strong-convergence algorithms, and \algoO yields only weak convergence.}

All algorithms were implemented in \texttt{Julia 1.12.6} and executed on a
PC equipped with an Intel Core\texttrademark~i9-9900K processor (3.6~GHz) and 32~GB of RAM.

{
\begin{remark}[Parameter selection]\label{remark:implementation}
Table~\ref{tab:sensitivity} reports a sensitivity analysis of Algorithm~\ref{algo:DIPCM} on a variant of the variational inclusion problem of Example~\ref{subsec:example1} with $W \sim U(0,1)$ (giving $L \approx 50$), $N = 200$, $5$ instances, and $\varepsilon = 10^{-6}$. The two parameters swept are $\bar\beta$ (cap for $\beta'_n$) and $\bar\theta$ (cap for $\theta_n$), at fixed $\alpha_n = 1/(n+1)$, $\sigma_n = 0.8 - \alpha_n$, $\varepsilon_n = 5/(n+1)^{2.1}$, and $\varepsilon'_n = 10/(n+1)^{2.1}$. All $32$ configurations achieve $100\%$ convergence with iteration counts ranging from $23$ to $128$. The bound $\bar\theta$ dominates on average: column means are $72$ ($\bar\theta = 0.5$), $52$ ($\bar\theta = 0.7$), $46$ ($\bar\theta = 0.9$), and $50$ ($\bar\theta = 0.95$). Within $\bar\theta = 0.9$, $\bar\beta \in [0.0, 0.3]$ yields $42$--$47$ iterations. We select $\bar\beta = 0.3$, $\bar\theta = 0.9$ (mean $42 \pm 2$), in a region where iteration counts are insensitive to small parameter changes. We also verify robustness to the $\varepsilon$ scale at the chosen $(\bar\beta, \bar\theta)$ by multiplying $\varepsilon_n$ and $\varepsilon'_n$ by factors in $\{0.01, 0.1, 1, 10, 100\}$: iteration counts are $1128$, $531$, $42$, $60$, and $63$ respectively, so the default scale is near the optimum, with mild slowdown for over-scaling and substantial slowdown for under-scaling. We therefore use the following parameter set throughout Section~\ref{S:5}: $\bar\beta = 0.3$, $\bar\theta = 0.9$, $\gamma = 1.1$, $\mu = 0.5$, $\alpha_n = 1/(n+1)$, $\sigma_n = 0.8 - \alpha_n$, $\varepsilon_n = 5/(n+1)^{2.1}$, $\varepsilon'_n = 10/(n+1)^{2.1}$, $\xi_n = 100/(n+1)^{1.1}$, and $\lambda_1 = 1/(1.05\,L)$ (overridden to $\lambda_1 = 0.05$ in Sections~\ref{subsec:example_2} and~\ref{subsec:elastic-net}).
The same stopping criterion is used for all algorithms within each example.
\end{remark}
}

{
\begin{table}[!ht]
\centering
\caption{Sensitivity analysis of Algorithm~\ref{algo:DIPCM} (\algoDIM\!\!) on Example~\ref{subsec:example1} with $W \sim U(0,1)$, $N = 200$, $5$ instances, $\varepsilon = 10^{-6}$. Fixed: $\alpha_n = 1/(n+1)$, $\sigma_n = 0.8 - \alpha_n$, $\varepsilon_n = 5/(n+1)^{2.1}$, $\varepsilon'_n = 10/(n+1)^{2.1}$, $\gamma = 1.1$, $\mu = 0.5$, $\lambda_1 = 1/(1.05\,L)$, $\xi_n = 100/(n+1)^{1.1}$. Average iterations and standard deviation are reported.}
\label{tab:sensitivity}
\begin{tabular}{c|rr|rr|rr|rr}
\hline
& \multicolumn{2}{c|}{$\bar\theta = 0.5$} & \multicolumn{2}{c|}{$\bar\theta = 0.7$} & \multicolumn{2}{c|}{$\bar\theta = 0.9$} & \multicolumn{2}{c}{$\bar\theta = 0.95$} \\
$\bar\beta$ & Avg & Std & Avg & Std & Avg & Std & Avg & Std \\
\hline
0.00 & 128 &  3 & 56 & 3 & 42 & 2 & 45 & 2 \\
0.05 & 111 &  3 & 57 & 3 & 47 & 4 & 45 & 2 \\
0.10 &  95 &  3 & 54 & 2 & 47 & 4 & 42 & 1 \\
0.15 &  79 &  2 & 55 & 2 & 47 & 5 & 41 & 2 \\
0.20 &  60 &  3 & 52 & 3 & 43 & 2 & 50 & 10 \\
0.25 &  35 & 11 & 49 & 4 & 42 & 2 & 55 & 10 \\
0.30 &  23 &  9 & 49 & 3 & 42 & 2 & 59 & 1 \\
0.40 &  46 &  5 & 41 & 3 & 61 & 2 & 60 & 1 \\
\hline
\end{tabular}
\end{table}
}

We implement the algorithms on the following numerical examples.

\subsection[Example: Variational Inclusion Problem]{Example: Variational Inclusion Problem \cite{dey_hybrid_2023}}
\label{subsec:example1}

Consider the variational inclusion problem:
\begin{equation}
\text{Find } x^\ast \in \mathbb{R}^n \text{ such that } {\mathbf{0}} \in A x^\ast + B x^\ast.
\end{equation}
where $A \in \R^{N\times N}$ is an upper triangular matrix with all entries equal to one,
and $B(x) = W^\top W x$ where $W$ is an $N\times N$ matrix with entries drawn independently 
from the uniform distribution $U(1, 100)$.
It is clear that $A$ is maximal monotone operator with resolvent 
$J^{A}_{\lambda}(x) = (I + \lambda A)^{-1}x$, and the operator $B$ is monotone and 
Lipschitz continuous with Lipschitz constant $L=\max(\operatorname{eig}(B))$. 
{We compare \algoT and \algoDIM with \algoD and \algoS.}
The problem dimension is taken as $N \in \{100, 150, 200, 250, 300\}$ 
and the prescribed error tolerances are $\epsilon \in \{10^{-1}, 10^{-2}, 10^{-3}, 10^{-4}, 10^{-5}, 10^{-6}\}$. 
We first generate a collection of $5$ independent problem instances; 
each instance consists of a matrix $W$ drawn as above and initial points $x_0,x_1 \in [0,1]^n$ 
with components sampled independently from the uniform distribution on the interval $[0,1]$. 
{ All four algorithms are tested on the same $5$ instances to ensure a fair comparison.} 
All methods are terminated when the residual norm satisfies
\[
\|y_n - w_n\| < \varepsilon
\]
or when the maximum number of iterations, set to $50{,}000$, is reached, whichever occurs first. 
{The numerical results are reported for each $(N,\varepsilon)$ and for each algorithm. The entries show the average CPU time (in seconds) and the average number of iterations over the $5$ instances. If an algorithm fails to reach the prescribed tolerance within the iteration limit on all $5$ instances, the corresponding table entry is marked by ``$-$''.
The results for $N=100, 200$ and $300$ are given
in Tables \ref{tab:comparison_n100},\ref{tab:comparison_n200}, and \ref{tab:comparison_n300}.}

{The parameters of \algoDIM (Algorithm~\ref{algo:DIPCM}) are
\[
\begin{array}{lclclclcl}
      \bar\beta &=& 0.3, & \bar\theta &=& 0.9, & \sigma_n &=& 0.8 - \alpha_n,\\[2pt]
      \mu &=& 0.5, & \gamma &=& 1.1, & \alpha_n &=& \dfrac{1}{n+1},\\[2pt]
      \lambda_1 &=& \dfrac{1}{1.05\,L}, & \xi_n &=& \dfrac{100}{(n+1)^{1.1}}, & & & \\[2pt]
      \varepsilon_n &=& \dfrac{5}{(n+1)^{2.1}}, & \varepsilon'_n &=& \dfrac{10}{(n+1)^{2.1}}. & & &
\end{array}
\]}
For \algoT  (Algorithm~\ref{algo:IPCMAS2})  the parameters are
\[
\begin{array}{lclclclclclcl}
      \mu &=& 0.5, & \gamma &=&  1.1, &\alpha &=&0.25,&\lambda_1&=&\dfrac{1}{1.05\,L}\\
      \theta &=& 0.5, &\xi_n&=&\dfrac{100}{(n+1)^{1.1}}.\\
\end{array}
\]
For \algoD, the parameters chosen are
\[
\begin{array}{lclclclclclclcl}
      \alpha &=& 0.5, &\gamma&=&1.5,& \alpha_n &=&  \alpha, &\tau_n &=&\dfrac{1}{n^2},&\lambda_n&=&\dfrac{1}{1.05\,L}\\
      \beta_n &=&  \dfrac{1}{5n+1}, &\theta_n &=&0.8-\beta_n.\\
\end{array}
\]
{ For \algoS~\cite{Suantai2024}, the parameters are chosen as in~\cite{Suantai2024}:
\[
\begin{array}{lclclclclcl}
      \gamma &=& 0.1, &\mu&=&0.9,&\lambda_0&=&0.01,\\
      \eta_k &=& \frac{1}{(k+1)^2}, &\delta_k &=&\frac{1}{(5k+2)^3}.\\
\end{array}
\]
}
To compare the performance of the {four} algorithms across the entire set of test problems, we adopt the performance profiles introduced by Dolan and Moré~\cite{dolan_benchmarking_2001}. {Given a set of problems $\mathcal{P}$ and solvers $\mathcal{S}$, the performance profile $\rho_s(\tau)$ of solver $s$ plots the fraction of problems for which the performance ratio $r_{p,s} = t_{p,s}/\min_{s' \in \mathcal{S}} t_{p,s'}$ satisfies $r_{p,s} \leq \tau$, where $t_{p,s}$ is the performance metric (time or iterations) of solver $s$ on problem $p$. A solver whose curve lies higher and further to the left is preferred.} 
Figures~\ref{fig:time_profile} and~\ref{fig:iter_profile} plot the performance profiles with respect to the CPU time and the number of iterations, respectively.

\begin{figure}
    \centering
    \includegraphics[width=0.5\linewidth]{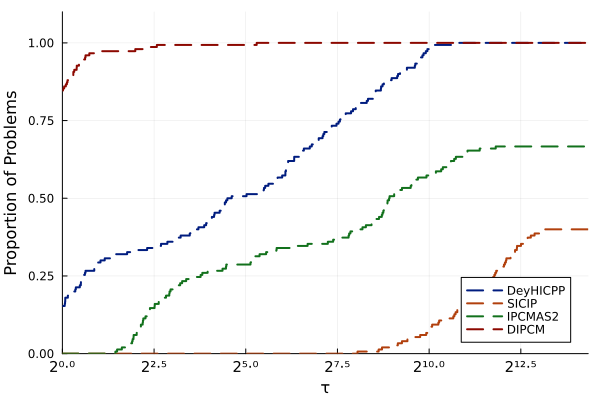}
    \caption{{Performance profile (CPU time) for Example~\ref{subsec:example1}: \algoDIM, \algoD, \algoS, \algoT over $5$ instances, $N \in \{100,\ldots,300\}$, $\varepsilon \in \{10^{-1},\ldots,10^{-6}\}$.}}
    \label{fig:time_profile}
\end{figure}

\begin{figure}
    \centering
    \includegraphics[width=0.5\linewidth]{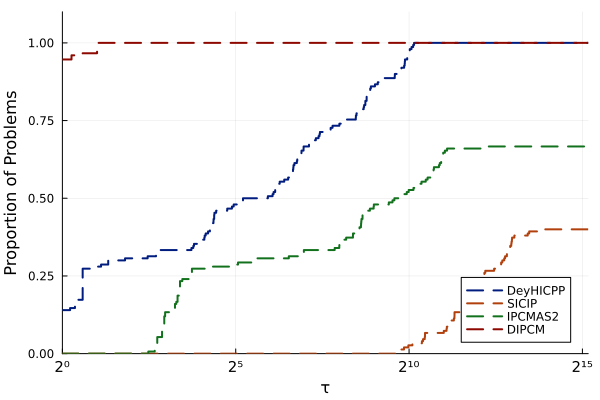}
    \caption{{Performance profile (iterations) for Example~\ref{subsec:example1}: same setup as Figure~\ref{fig:time_profile}.}}
    \label{fig:iter_profile}
\end{figure}

\begin{table}[htbp]
\centering
{
\caption{Median CPU time (seconds) and average iterations over 5 instances, $N=100$.}
\label{tab:comparison_n100}
\begin{tabular}{c*{4}{rr}}
\hline
\multicolumn{9}{c}{$N=100$}\\
\hline
  & \multicolumn{2}{c|}{\textbf{\algoD}} & \multicolumn{2}{c|}{\textbf{\algoS}} & \multicolumn{2}{c|}{\textbf{\algoT}} & \multicolumn{2}{c}{\textbf{\algoDIM}}\\
\cline{2-9}
Error & Time & No.\ It. & Time & No.\ It. & Time & No.\ It. & Time & No.\ It.\\
\hline
$10^{-1}$ & 1.80e-04 & 3 & 0.0708 & 1873 & 6.35e-04 & 19 & 2.10e-04 & 2 \\
$10^{-2}$ & 4.72e-04 & 12 & 0.2737 & 8274 & 0.0084 & 194 & 4.05e-04 & 6 \\
$10^{-3}$ & 0.0051 & 147 & 0.9434 & 24936 & 0.1049 & 2534 & 6.95e-04 & 10 \\
$10^{-4}$ & 0.0391 & 997 & -- & $-$ & 0.4922 & 14468 & 9.27e-04 & 15 \\
$10^{-5}$ & 0.1624 & 4810 & -- & $-$ & -- & $-$ & 0.0019 & 34 \\
$10^{-6}$ & 0.9235 & 26536 & -- & $-$ & -- & $-$ & 0.0035 & 57 \\
\hline
\end{tabular}
}
\end{table}

\begin{table}[htbp]
\centering
{
\caption{Median CPU time (seconds) and average iterations over 5 instances, $N=200$.}
\label{tab:comparison_n200}
\begin{tabular}{c*{4}{rr}}
\hline
\multicolumn{9}{c}{$N=200$}\\
\hline
  & \multicolumn{2}{c|}{\textbf{\algoD}} & \multicolumn{2}{c|}{\textbf{\algoS}} & \multicolumn{2}{c|}{\textbf{\algoT}} & \multicolumn{2}{c}{\textbf{\algoDIM}}\\
\cline{2-9}
Error & Time & No.\ It. & Time & No.\ It. & Time & No.\ It. & Time & No.\ It.\\
\hline
$10^{-1}$ & 4.39e-04 & 3 & 0.7948 & 7390 & 0.0032 & 21 & 3.47e-04 & 2 \\
$10^{-2}$ & 5.78e-04 & 6 & 2.83 & 23829 & 0.0049 & 40 & 5.93e-04 & 6 \\
$10^{-3}$ & 0.0200 & 204 & -- & $-$ & 0.4173 & 3843 & 9.54e-04 & 10 \\
$10^{-4}$ & 0.1580 & 1476 & -- & $-$ & 2.91 & 25010 & 0.0025 & 15 \\
$10^{-5}$ & 0.7962 & 7592 & -- & $-$ & -- & $-$ & 0.0033 & 25 \\
$10^{-6}$ & 4.88 & 37448 & -- & $-$ & -- & $-$ & 0.0055 & 41 \\
\hline
\end{tabular}
}
\end{table}

\begin{table}[htbp]
\centering
{
\caption{Median CPU time (seconds) and average iterations over 5 instances, $N=300$.}
\label{tab:comparison_n300}
\begin{tabular}{c*{4}{rr}}
\hline
\multicolumn{9}{c}{$N=300$}\\
\hline
  & \multicolumn{2}{c|}{\textbf{\algoD}} & \multicolumn{2}{c|}{\textbf{\algoS}} & \multicolumn{2}{c|}{\textbf{\algoT}} & \multicolumn{2}{c}{\textbf{\algoDIM}}\\
\cline{2-9}
Error & Time & No.\ It. & Time & No.\ It. & Time & No.\ It. & Time & No.\ It.\\
\hline
$10^{-1}$ & 0.0012 & 3 & 3.90 & 13120 & 0.0065 & 23 & 0.0017 & 3 \\
$10^{-2}$ & 0.0035 & 7 & 11.57 & 40694 & 0.0097 & 39 & 0.0024 & 6 \\
$10^{-3}$ & 0.0537 & 221 & -- & $-$ & 1.43 & 4869 & 0.0029 & 9 \\
$10^{-4}$ & 0.5684 & 1828 & -- & $-$ & 9.57 & 31362 & 0.0061 & 14 \\
$10^{-5}$ & 2.68 & 9191 & -- & $-$ & -- & $-$ & 0.0073 & 22 \\
$10^{-6}$ & 13.06 & 44418 & -- & $-$ & -- & $-$ & 0.0146 & 44 \\
\hline
\end{tabular}
}
\end{table}

{ From the tables and performance profiles, we draw the following observations.
\algoDIM is the best-performing method across all dimensions and tolerances. At $N=300$ and $\varepsilon = 10^{-6}$, \algoDIM requires $44$ iterations versus $44{,}418$ for \algoD---over a $1{,}000$-fold improvement.
\algoD converges at all tolerances but requires orders of magnitude more iterations.
\algoT converges at coarser tolerances ($\varepsilon \geq 10^{-4}$) but fails at $\varepsilon \leq 10^{-5}$.
\algoS converges only at $\varepsilon \geq 10^{-3}$ for $N=100$ and $\varepsilon \geq 10^{-2}$ for $N \geq 200$, due to its conservative parameters and summable inertial sequences.}

\subsection[Example: Split Feasibility Problem]{Example: Split Feasibility Problem (SFP) in $L^2[0,1]$} \label{subsec:example_2}

Let $\mathcal{H}_1$ and $\mathcal{H}_2$ be two Hilbert spaces and 
$M:\mathcal{H}_1\to \mathcal{H}_2$ be a bounded linear operator. 
The Split Feasibility Problem (SFP) \cite{censor_multiprojection_1994,cegielski_iterative_2013} 
is described as follows: 
\begin{equation}\label{eq:sfp_g}    
\text{Find} \quad x^\ast \in  \mathcal{C}\subseteq \mathcal{H}_1 
\quad \text{such that} \quad Mx^\ast \in \mathcal{Q}\subseteq \mathcal{H}_2. 
\end{equation}
where $\mathcal{C}$ and $\mathcal{Q}$ are closed convex sets. This problem appears in signal processing, particularly phase retrieval and other image restoration problems. In this example, we take $\mathcal{H}_1=\mathcal{H}_2=L^2[0,1]$ and the operator $M$ to be the identity. The two sets $\mathcal{C}$ and $\mathcal{Q}$ are defined as (see \cite{dey_hybrid_2023})
\begin{equation}
\begin{array}{lclllcl}
    \mathcal{C} &=& \{x(t) \in L^2[0,1] : \int_{0}^{1} 3 t^2\,x(t)\,dt = 0 \}, &\qquad 
    & \mathcal{Q} &=& \{x(t) \in L^2[0,1] : \int_{0}^{1} \frac{t}{3}\,x(t)\,dt \geq -1\}.
\end{array} 
\end{equation}

The set $\mathcal{C}$ is a hyperplane and the set $\mathcal{Q}$ is a halfspace. 
Then, the problem \eqref{eq:sfp_g} can be rewritten as
\begin{equation}
\label{eq:sfp_s}    
\text{Find} \quad x^\ast \in  \mathcal{C} \quad \text{such that} \quad x^\ast \in \mathcal{Q}. 
\end{equation}

{ A point $x^\ast$ is a solution of~\eqref{eq:sfp_s}} if and only if 
\[
\displaystyle x^\ast \in \operatorname{argmin}_{x\in \mathcal{C}}~ g(x):=\frac{1}{2}\|x-P_{\mathcal{Q}}(x)\|^2.
\]
Here $P_{\mathcal{Q}}$ is the metric projection onto $\mathcal{Q}$ given by
\[
\displaystyle P_{\mathcal{Q}}(x) = \operatorname{argmin}_{y\in \mathcal{Q}}~ \|x-y\|^2.
\]
The function $g$ is convex, continuous and differentiable with a Lipschitz continuous gradient $\nabla g(x) = (I-P_{\mathcal{Q}})(x)$ (see \cite{rockafellar_variational_1998}). 
Furthermore, $\nabla g$ is $1-$inverse strongly monotone \cite{dey_hybrid_2023}; that is
\[
\langle\nabla g(x)-\nabla g(y), x-y\rangle \geq \|\nabla g(x)-\nabla g(y)\|^2, \quad x, y \in L^2[0,1].
\]
This shows that $x^\ast$ solves problem \eqref{eq:sfp_s} if and only if
it solves the following variational inclusion {problem (see, e.g., \cite{Ceng2012})}: 
\begin{equation}
\label{eq:example2_vip}    
\text{Find} \quad x \in  L^2[0,1] \quad \text{such that} \quad 0 \in \partial \chi_{\mathcal{C}}(x) ~+~ \nabla g(x),
\end{equation}
where $\partial \chi_{\mathcal{C}}$ is the subdifferential of the indicator function $\chi_{\mathcal{C}}$ which is defined as

\[
\chi_{\mathcal{C}}(x) = 
\left\{
\begin{array}{lcl}
     0, &\text{if}, & x\in \mathcal{C}  \\
     +\infty, & \text{if}, & x\notin \mathcal{C}.
\end{array}\right.
\]
So, we have $A=\partial \chi_{\mathcal{C}}$, $B=\nabla g$, and therefore, $J^{A}_{\lambda}=P_{\mathcal{C}}$.

{We solve problem \eqref{eq:example2_vip} by \algoDIM, \algoT, \algoD, and \algoS. For \algoDIM we use the parameters of Remark~\ref{remark:implementation}; for \algoT and \algoS we use the same parameters as in Section~\ref{subsec:example1}. The parameters of \algoD\ are}
\[
\left\{
     \begin{array}{lclclclclclclcl}
      \alpha &=& 0.5, &\gamma&=&0.01,& \alpha_n &=&  \alpha, &\tau_n &=&\frac{1}{n^2},&\lambda_n&=&0.01\\
      \beta_n &=&  \frac{1}{\sqrt{n+1}}, &\theta_n &=&0.8-\beta_n.\\
\end{array}
\right.
\]
The initial points used for {all} algorithms are listed in Table \ref{tab:ex2_initials}. 
To proceed, we consider the inner product on $L^2[0,1]$ as 
\begin{equation}
\label{eq:inner_product}  
\langle x, y \rangle = \int_0^1 x(t) y(t) dt
\end{equation}
and the norm 
\[
\|x\|_{L^2[0,1]} = \left(\int_0^1 |x(t)|^2  dt\right)^{1/2}.
\]
The projections $P_{\mathcal{C}}$ and $P_{\mathcal{Q}}$ are defined as 
(see \cite{cegielski_iterative_2013}) 
\[
P_{\mathcal{C}}(x(t))= \begin{cases}x(t)-\frac{\left\langle x(t), 3 t^2\right\rangle}{\left\|3 t^2\right\|^2_{L^2[0,1]}} 3 t^2, & \text { if }\left\langle x(t), 3 t^2\right\rangle \neq 0 \\ x(t), & \text { if }\left\langle x(t), 3 t^2\right\rangle=0\end{cases}
\]
and
$$
P_{\mathcal{Q}}(x(t))= \begin{cases}x(t)-\frac{\langle x(t),t / 3\rangle~+~1}{\|t / 3\|_{L^2[0,1]}}(t / 3), & \text { if }\langle x(t), t / 3\rangle<-1, \\ x(t), & \text { if }\langle x(t), t / 3\rangle \geq-1 .\end{cases}
$$
The integral in \eqref{eq:inner_product} is evaluated using Simpson's rule with $100$ subintervals. The stopping criterion is, as in \cite{dey_hybrid_2023},
\[
E_n = \frac{1}{2} \left\| x_n(t) - P_{\mathcal{C}}(x_n(t))\right\|_{L^2[0,1]}^2 +
\frac{1}{2} \left\| x_n(t) - P_{\mathcal{Q}}(x_n(t))\right\|_{L^2[0,1]}^2 ~ < 10^{-3},
\]
{or a maximum of $50{,}000$ iterations, whichever occurs first.}

\begin{table}[!ht]
\centering
\begin{tabular}{c l l}
\hline
Instance & $x_0(t)$ & $x_1(t)$ \\
\hline
1 & $t^{3} e^{t}/211 + 5t$ & $\sin(t) + t^{6}$ \\
2 & $e^{t}$ & $t\, e^{t^{3}}$ \\
3 & $t + 1$ & $3t^{2} + t$ \\
4 & $11 \sin(t)$ & $5t^{2}$ \\
5 & $15 t^{3} + e^{t}/22$ & $\sin\!\left(\tfrac{t}{2}\right)$ \\
6 & $e^{t}$ & $\cos(2\pi t)$ \\
7 & $t + 1$ & $3t^{3} + 2t$ \\
8 & $11 \sin(t)$ & $\sqrt{t}$ \\
\hline
\end{tabular}
\caption{Test instances and their functions $x_0(t)$ and $x_1(t)$.}
\label{tab:ex2_initials}
\end{table}

\begin{table}[htbp]
\centering
{
\begin{tabular}{c*{4}{rr}}
\hline
\multirow{2}{*}{Instance} & \multicolumn{2}{c|}{\textbf{\algoD}} & \multicolumn{2}{c|}{\textbf{\algoS}} & \multicolumn{2}{c|}{\textbf{\algoT}} & \multicolumn{2}{c}{\textbf{\algoDIM}} \\
 & Time & No.\ It. & Time & No.\ It. & Time & No.\ It. & Time & No.\ It. \\
\hline
1 & 9.66e-05 & 5 & 2.66e-04 & 34 & 1.60e-04 & 16 & 9.96e-05 & 3 \\
2 & 7.87e-05 & 6 & 3.77e-04 & 50 & 1.46e-04 & 17 & 8.19e-05 & 4 \\
3 & 4.61e-05 & 3 & 3.85e-04 & 58 & 1.51e-04 & 19 & 5.99e-05 & 4 \\
4 & 4.33e-05 & 3 & 3.65e-04 & 54 & 1.46e-04 & 20 & 5.75e-05 & 4 \\
5 & 4.63e-05 & 4 & 3.55e-04 & 54 & 1.06e-04 & 14 & 4.39e-05 & 3 \\
6 & 4.43e-05 & 4 & 2.53e-04 & 40 & 8.46e-05 & 11 & 5.02e-05 & 4 \\
7 & 6.48e-05 & 7 & 3.68e-04 & 60 & 1.39e-04 & 20 & 5.27e-05 & 4 \\
8 & 9.19e-05 & 5 & 3.55e-04 & 49 & 1.14e-04 & 16 & 4.74e-05 & 3 \\
\hline
Median/Avg & 5.56e-05 & 4.6 & 3.60e-04 & 49.9 & 1.42e-04 & 16.6 & 5.51e-05 & 3.6 \\
\hline
\end{tabular}
}
\caption{{Median CPU time (seconds) and average iterations on the SFP ($L^2[0,1]$, $100$ subintervals, $\varepsilon = 10^{-3}$); per-instance initial points listed in Table~\ref{tab:ex2_initials}.}}
\label{tab:example2_comparizon}
\end{table}

{
  Table~\ref{tab:example2_comparizon} reports the numerical results for the split feasibility problem over the eight test instances listed in
  Table~\ref{tab:ex2_initials}. Algorithm~\ref{algo:DIPCM} (\algoDIM\!\!) converges in $3$--$4$ iterations, with an average of $3.6$, while \algoD
  requires $3$--$7$ iterations (average $4.6$) and \algoT requires $11$--$20$ iterations (average $16.6$). In contrast, \algoS requires substantially more iterations, ranging from $34$ to $60$, with an
  average of $49.9$. In terms of CPU time, \algoDIM and \algoD\ perform almost identically, with median times $5.51\times 10^{-5}$ and $5.56\times
  10^{-5}$ seconds, respectively; \algoT is moderately slower at $1.42\times 10^{-4}$ seconds, and \algoS is clearly slowest at $3.60\times 10^{-4}$ seconds. These results show that, for this SFP
  example with $L=1$, \algoDIM provides a modest improvement in iteration count over \algoD, with the other three methods all markedly more efficient than
  \algoS. The stronger advantage of \algoDIM becomes more apparent on problems with larger Lipschitz constant, as observed in
  Example~\ref{subsec:example1}.
  }

\subsection{Example: Elastic Net Regularization}\label{subsec:elastic-net}

Linear problems involving large, ill-conditioned matrices arise in
engineering, physics, and medicine \cite{zou_regularization_2005, huang_traction_2019}. 
We consider the linear systems of the form 
\begin{equation}
    \label{eq:linear_system}
    y = Xw + \varepsilon
\end{equation}
with $X \in \R^{m\times n}$ and noise vector $\varepsilon \in \R^m$. 
Typically, $X$ has a large condition number. A common remedy is to add $\ell_1$ and $\ell_2$ penalties to the least-squares residual, leading to the following minimization problem.
\begin{equation}\label{eq:naive-elastic-net}
  \min_{w\in\mathbb{R}^{n}}~ \frac{1}{2}\,\|Xw - y\|_2^2 \;+\; \varrho_1\|w\|_1 \;+\; \frac{\varrho_2}{2}\,\|w\|_2^2.
\end{equation}

The $\ell_1$ and $\ell_2$ are norms on $\R^n$ given by
\[
\|w\|_1 = \sum_{i=1}^n |w_i|, \qquad \|w\|_2 = \sqrt{\sum_{i=1}^n w^2_i}.
\]
The parameters $\varrho_1$ and $\varrho_2$ are called the regularization parameters. With $\varrho_1>0$ and $\varrho_2>0$, problem \eqref{eq:naive-elastic-net} is known as the naive elastic net (naive-EN) regularization problem \cite{zou_regularization_2005}. By setting 
\[
\Tilde{y} = \begin{bmatrix}
    y\\ 
    \mathbf{0}
\end{bmatrix}\in \R^{(m+n)}, \quad \Tilde{X} = \frac{1}{\sqrt{1+\varrho_2}}\begin{bmatrix}
    X \\
    \sqrt{\varrho_2}\mathbf{I} 
\end{bmatrix} \in \R^{(m+n)\times n}, \quad \Tilde{w} = \sqrt{1+\varrho_2}~w,
\]
Problem \eqref{eq:naive-elastic-net} becomes
\begin{equation}\label{eq:elastic-net-0}
  \min_{\Tilde{w}\in\mathbb{R}^{n}}~ \frac{1}{2}\,\|\Tilde{X}\Tilde{w} - \Tilde{y}\|_2^2 \;+\; \frac{\varrho_1}{\sqrt{1+\varrho_2}}\|\Tilde{w}\|_1.
\end{equation}
Further rescaling $\Tilde{w}$ by factor of $\sqrt{1+\varrho_2}$, we  reach the so-called elastic net (EN) regularization problem, see \cite{zou_regularization_2005} for more details and justifications, 
\begin{equation}\label{eq:elastic-net}
  \min_{\Tilde{w}\in\mathbb{R}^{n}}~ \frac{1}{2}\,\|\Tilde{X}\Tilde{w} - \Tilde{y}\|_2^2 \;+\; \frac{\varrho_1}{1+\varrho_2}\|\Tilde{w}\|_1.
\end{equation}
Elastic-net regularization has been shown to outperform $\ell_1$ and $\ell_2$ alternatives in several settings~\cite{huang_traction_2019}.

Problem \eqref{eq:elastic-net} can be reformulated as the monotone inclusion; that is

\[
\text{find  }\Tilde{w} \text{ such that }
  0 \in A(\Tilde{w}) + B(\Tilde{w}),
\]
with
$
  A=\partial\!\left(\frac{\varrho_1}{1+\varrho_2}\|\cdot\|_1\right),\quad \text{and}\quad
  B(\Tilde{w})=\nabla \left(\frac{1}{2}\,\|\Tilde{X}\Tilde{w} - \Tilde{y}\|_2^2\right) = \Tilde{X}^\top (\Tilde{X}\Tilde{w} - \Tilde{y}).
$
Moreover, the resolvent, $J_A^{\lambda}(\Tilde{w}) = (\zeta_i)_{i=1,\cdots,n}$, 
where 
$$
\zeta_i = \operatorname{sign}(\Tilde{w}_i)\max\left\{0, |\Tilde{w}_i|- \frac{\lambda\varrho_1}{1+\varrho_2}\right\},
$$ (see~\cite{vanhieuIterativeRegularizationMethods2022}).
The operator $A$ is a maximal monotone operator~\cite{Rockafellar1970MaxMonSubdiff} and $B$ is monotone and Lipschitz continuous with Lipschitz constant $L=\|\Tilde{X}\|^2$.

Following the framework of Zou and Hastie~\cite{zou_regularization_2005}, we conduct a simulation study. We generate data from the model 
\[
\Tilde{y}=\Tilde{X} \Tilde{\mathbf{w}}+\sigma \boldsymbol{\varepsilon},
\]
where $\boldsymbol{\varepsilon} \sim N(0, 1)$ (normal distribution with mean $0$ and standard deviation $1$) with $\sigma=3$. The design matrix $\Tilde{X}$ is generated  such that the pairwise correlation between $\Tilde{X}_{:,i}$ and $\Tilde{X}_{:,j}$ is equal to $(0.5)^{|i-j|}$; that is 
\[
\operatorname{corr}\left(\Tilde{X}_{:,i},\Tilde{X}_{:,j}\right)=(0.5)^{|i-j|}.
\]
The true coefficient vector is set to $\Tilde{w}^\ast=(3,1.5,0,0,2,0,0, 0 )$, yielding a sparse signal with three non-zero components out of  $8$ total covariates.

For each simulation run, we generate a dataset partitioned into training ($N_{\text{train}} = 20$), validation ($N_{\text{val}} = 20$), and test ($N_{\text{test}} = 200$) sets. We employ 10-fold cross-validation on the training data to select the optimal tuning parameters $(\varrho_1, \varrho_2)$ from predefined grids: $\varrho_1 \in \{0.0, 0.01, 0.1, 1.0, 10.0, 100.0\}$ and $\varrho_2 \in \{0.01, 0.1, 1.0, 10.0, 100.0\}$. The model with the lowest cross-validation error is then retrained on the complete training set and evaluated on the held-out test set. We replicate this procedure 50 times with different random seeds to assess variability across runs.

{We compare four algorithms: \algoDIM, \algoT, \algoD, and \algoS.} The stopping criterion is that of Example \ref{subsec:example1} with tolerance $10^{-6}$ and maximum number of iterations $50{,}000$.

We report test mean squared error (MSE), coefficient estimation error $\|\Tilde{w} - \Tilde{w}^{\ast}\|$, and variable selection metrics (precision, recall, F1-score), summarized as mean $\pm$ standard deviation over the $50$ runs.

\vspace{2mm}

{Figure~\ref{fig:combined_results} compares the four algorithms across four metrics; Tables~\ref{tab:prediction_results}, \ref{tab:selection_results}, and \ref{tab:timing_results} report prediction accuracy, variable selection, and timing, respectively.}

\begin{figure}
    \centering
    \includegraphics[width=\linewidth]{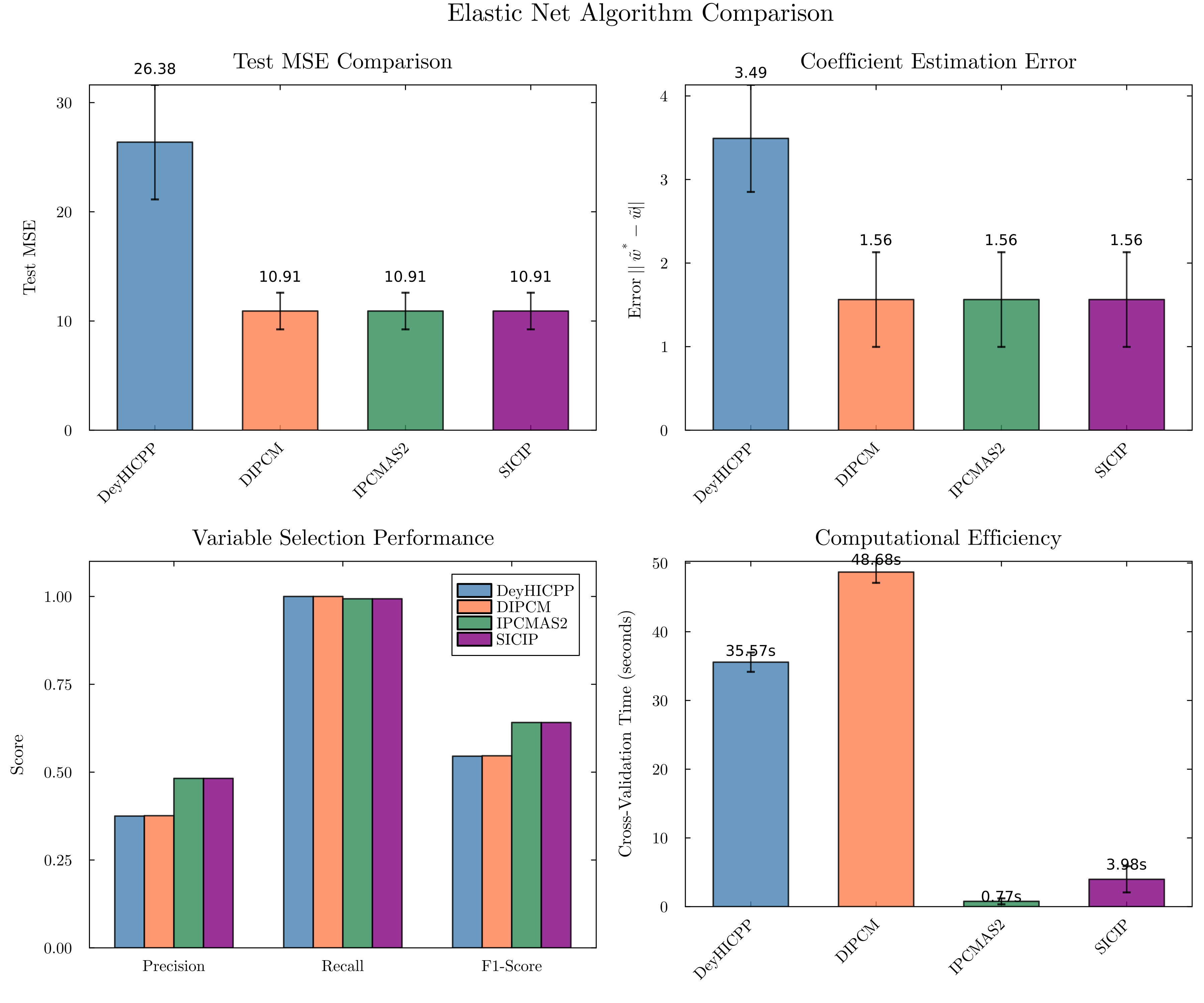}
    \caption{{Elastic net comparison: test MSE, coefficient estimation error, variable selection, and computational efficiency for \algoDIM, \algoT, \algoS, and \algoD ($50$ runs).}}
    \label{fig:combined_results}
\end{figure}
\begin{table}[htbp]
\centering
{
\caption{Prediction accuracy and coefficient estimation performance (mean $\pm$ standard deviation over 50 runs).}
\label{tab:prediction_results}
\begin{tabular}{lccc}
\toprule
\textbf{Algorithm} & \textbf{Test MSE} & \textbf{Coefficient Error} & \textbf{Relative Error} \\
\midrule
\algoDIM & $10.9124 \pm 1.6782$  & $1.5632 \pm 0.5664$ & $0.4003 \pm 0.1450$ \\
\algoT   & $10.9126 \pm 1.6784$  & $1.5633 \pm 0.5666$ & $0.4003 \pm 0.1451$ \\
\algoS   & $10.9126 \pm 1.6784$  & $1.5633 \pm 0.5666$ & $0.4003 \pm 0.1451$ \\
\algoD   & $26.3782 \pm 5.2435$  & $3.4918 \pm 0.6403$ & $0.8942 \pm 0.1640$ \\
\bottomrule
\end{tabular}
}
\end{table}

\begin{table}[htbp]
\centering
{
\caption{Variable selection performance (mean $\pm$ standard deviation over 50 runs).}
\label{tab:selection_results}
\begin{tabular}{lcccc}
\toprule
\textbf{Algorithm} & \textbf{Precision} & \textbf{Recall} & \textbf{F1-Score} & \textbf{Median Test MSE} \\
\midrule
\algoDIM & $0.38 \pm 0.01$ & $1.00 \pm 0.00$ & $0.55 \pm 0.01$ & $10.72$ \\
\algoT   & $0.48 \pm 0.12$ & $0.99 \pm 0.05$ & $0.64 \pm 0.10$ & $10.72$ \\
\algoS   & $0.48 \pm 0.12$ & $0.99 \pm 0.05$ & $0.64 \pm 0.10$ & $10.72$ \\
\algoD   & $0.38 \pm 0.00$ & $1.00 \pm 0.00$ & $0.55 \pm 0.00$ & $27.33$ \\
\bottomrule
\end{tabular}
}
\end{table}

\begin{table}[!ht]
\centering
{
\caption{Computational efficiency (mean $\pm$ standard deviation over 50 runs).}
\label{tab:timing_results}
\begin{tabular}{lccc}
\toprule
\textbf{Algorithm} & \textbf{CV Time (s)} & \textbf{Mean Training Time (s)} & \textbf{Total Training Time (s)} \\
\midrule
\algoDIM & $48.68 \pm 1.57$ & $0.180 \pm 0.065$ & $8.98$ \\
\algoT   & $0.77 \pm 0.43$  & $0.001 \pm 0.001$ & $0.03$ \\
\algoS   & $3.98 \pm 1.91$  & $0.001 \pm 0.002$ & $0.07$ \\
\algoD   & $35.57 \pm 1.41$ & $0.129 \pm 0.054$ & $6.45$ \\
\bottomrule
\end{tabular}
}
\end{table}

{\algoDIM, \algoT, and \algoS achieve nearly identical prediction accuracy (test MSE $\approx 10.91$, differing by less than $10^{-3}$), all outperforming \algoD (MSE $26.38$). This near-coincidence is consistent with the three algorithms reaching the same elastic-net minimizer---unique because the objective is strongly convex when $\varrho_2 > 0$---within the convergence tolerance $10^{-6}$. The discrepancy in variable selection (Table~\ref{tab:selection_results}) reflects threshold sensitivity at $|\hat w_i| \approx 10^{-6}$: \algoT and \algoS occasionally classify a near-zero coefficient as inactive while \algoDIM retains it. On computational cost (Table~\ref{tab:timing_results}), \algoT is the fastest method (CV time $0.77$~s, vs.\ $3.98$ for \algoS, $35.57$ for \algoD, and $48.68$ for \algoDIM), reflecting its $R$-linear convergence under strong monotonicity (verified empirically in Section~\ref{sub:r_linearity}). \algoDIM is slower on this low-dimensional, strongly convex problem; its strength is most apparent on higher-dimensional, only-monotone VIPs (Example~\ref{subsec:example1}), where it is over $10^{3}$ times faster than the alternatives.}

\subsection[Experiment: Verification of R-Linear Convergence]{Experiment: Verification of $R$-Linear Convergence}
\label{sub:r_linearity}
{To verify the R-linear convergence of \algoT\ (Section~\ref{S:4}), we run $20$ independent simulations on the elastic net problem and fit a linear model to $\log(\|w_k - w^*\|)$ versus $k$.}
We use $\rho$ to denote the R-linear convergence rate and $\mathcal{R}^2$ to denote the regression goodness-of-fit.
\begin{table}[!ht]
\centering
\caption{{R-linear convergence verification ($20$ runs, elastic net, $\mathcal{R}^2 > 0.95$ criterion).}}
\label{tab:linear_convergence}
\begin{tabular}{lccc}
\toprule
\textbf{Algorithm} & \textbf{Rate $\rho$} & \textbf{$\mathcal{R}^2$} & \textbf{Linear} \\
\midrule
{\algoDIM} & {$1.000 \pm 0.000$} & {$0.767$} & {$0\%$} \\
\algoD    & $1.000 \pm 0.000$ & $0.945$ & $10\%$  \\
{\algoS}    & {$0.996 \pm 0.002$} & {$0.980$} & {$90\%$}  \\
\textbf{\algoT} & $\mathbf{0.983 \pm 0.008}$ & $\mathbf{0.989}$ & $\mathbf{95\%}$ \\
\bottomrule
\end{tabular}
\end{table}

{\algoT exhibits the fastest R-linear convergence ($\rho = 0.983$, $\mathcal{R}^2 = 0.989$, $95\%$ of runs linear). \algoS also exhibits R-linear behavior at a slower rate ($\rho = 0.996$, $\mathcal{R}^2 = 0.980$, $90\%$ linear). Neither \algoDIM nor \algoD show R-linear behavior ($\rho \approx 1$, $\mathcal{R}^2 < 0.95$). \algoDIM's strong convergence (Theorem~\ref{thm:halpern-strong}) is not R-linear in general. Figure~\ref{fig:ipcmas2_convergence} shows the detailed convergence analysis for~\algoT.}

\begin{figure}[H]
    \centering
    \includegraphics[width=\linewidth]{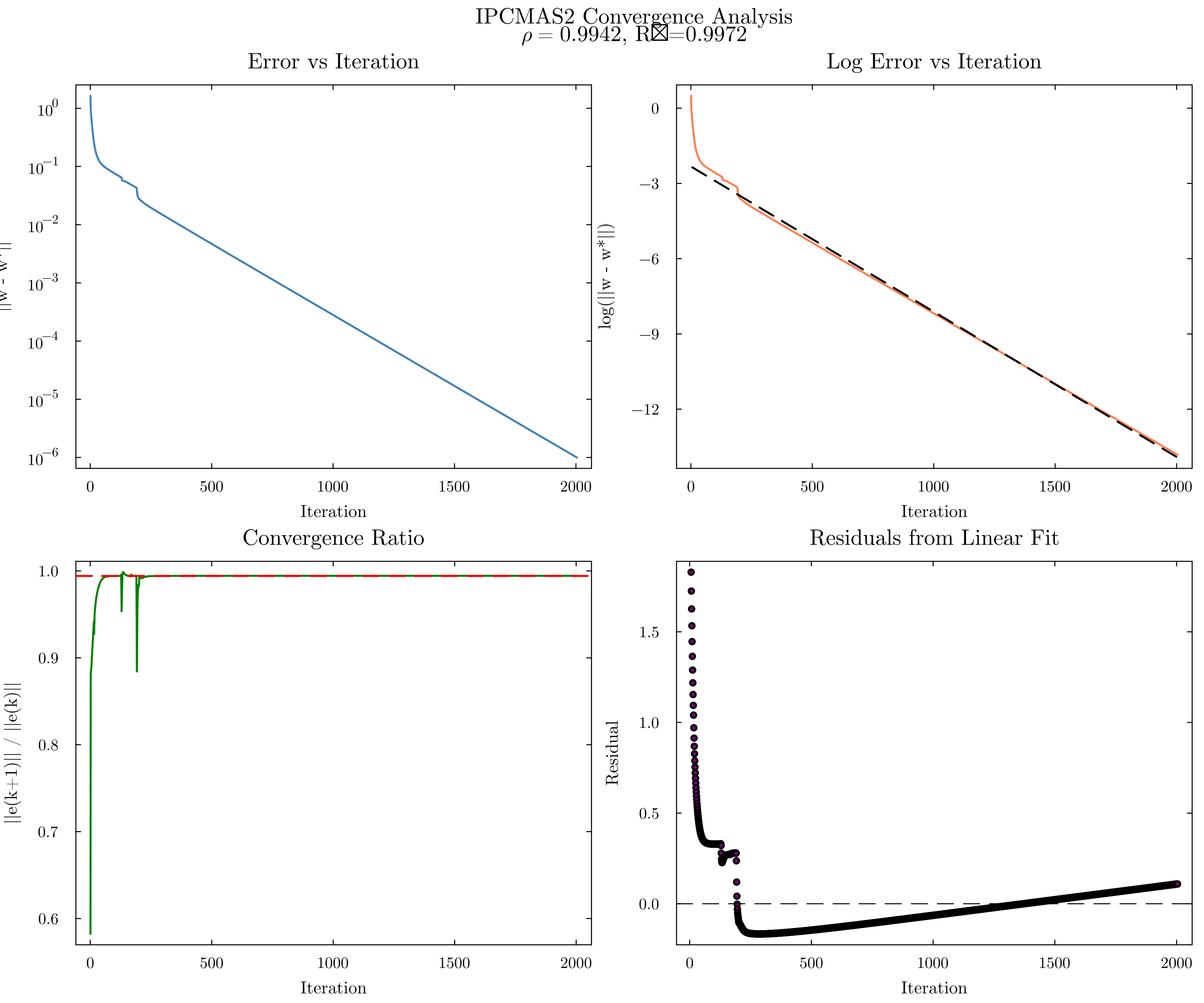}
    \caption{{Convergence analysis of \algoT on elastic net: error decay, log-error regression ($\mathcal{R}^2 = 0.989$), ratio $\|e_{k+1}\|/\|e_k\| \to 0.983$, and residuals.}}
    \label{fig:ipcmas2_convergence}
\end{figure}

{
\section{Conclusion}\label{S:6}

We studied the variational inclusion problem in a real Hilbert space and proposed three algorithms with inertial effects. The proposed algorithms combine a double inertial scheme with an adaptive stepsize that does not require knowledge of the Lipschitz constant of the operator. For these algorithms, we established weak convergence, strong convergence, and $R$-linear convergence under appropriate assumptions. We also derived a nonasymptotic convergence rate estimate for approximate solutions.

The proposed algorithms contain different variants corresponding to distinct convergence behaviors, namely weak, strong, and $R$-linear convergence. In particular, one variant ensures weak convergence, another achieves strong convergence without assuming strong monotonicity, while a modified variant attains $R$-linear convergence under strong monotonicity. The numerical results further confirm the efficiency and competitiveness of the proposed methods compared to existing algorithms.

For future work, it would be interesting to study stochastic versions of the proposed methods, as well as inexact evaluations of the resolvent, since computing the resolvent of a set-valued operator may be difficult in some applications.
}

{
\section*{Declarations}
\textbf{Conflict of interest:} The authors declare that they have no conflict of interest.

\section*{Data Availability}
No external data were used in this study. All numerical results were generated by the algorithms described in the paper. The Julia source code that reproduces the experiments in Section~\ref{S:5} is available from the corresponding author upon reasonable request.
}

\section*{Funding}
This work was supported by King Fahd University of Petroleum \& Minerals (KFUPM) through project No.~IN26075.

{
\section*{AI Use Declaration}
During the preparation of this work the authors used Claude (Anthropic) in order to assist with manuscript editing, including tightening prose, drafting and revising portions of the numerical-experiments narrative, verifying LaTeX formatting, checking internal consistency of cross-references and notation, and editing experiment scripts. After using this tool, the authors reviewed and edited the content as needed and take full responsibility for the content of the publication.
}

\section*{Acknowledgment}
The authors thank King Fahd University of Petroleum \& Minerals (KFUPM) and the Interdisciplinary Research Center for Smart Mobility and Logistics (IRC-SML), KFUPM, for institutional support.

 \printbibliography

\end{document}